\documentclass[11pt]{amsart}
\usepackage{amsmath}
\usepackage{amsmath,amssymb,enumitem,ulem,multicol}
\usepackage{mathabx}
\usepackage{amssymb}
\usepackage{amsthm}
\usepackage{parskip}
\usepackage{setspace}
\usepackage{mathrsfs}
\usepackage{mathtools}
\usepackage{float}
\usepackage[utf8]{inputenc}
\DeclareUnicodeCharacter{2212}{-}
\usepackage[english]{babel}
\usepackage{graphicx}
\usepackage{subcaption}
\theoremstyle{remark}

\def\R{\mathbb{R}}

\def\cN{\mathcal{N}}
\def\cA{\mathcal{A}}
\def\cK{\mathcal{K}}
\def\cT{\mathcal{T}}
\def\cE{\mathcal{E}}

\def\p{\partial}

\def\[{\partial}
\def\O{\Omega}
\def\ssT{{\scriptscriptstyle T}}
\def\HT{{H^2(\O,\cT_h)}}
\def\mean#1{\left\{\hskip -5pt\left\{#1\right\}\hskip -5pt\right\}}
\def\jump#1{\left[\hskip -3.5pt\left[#1\right]\hskip -3.5pt\right]}
\def\smean#1{\{\hskip -3pt\{#1\}\hskip -3pt\}}
\def\sjump#1{[\hskip -1.5pt[#1]\hskip -1.5pt]}
\def\jumptwo{\jump{\frac{\p^2 u_h}{\p n^2}}}
\def\b#1{\boldsymbol{#1}}
\def\norm #1{{\left\vert\kern-0.25ex\left\vert\kern-0.25ex\left\vert #1 
    \right\vert\kern-0.25ex\right\vert\kern-0.25ex\right\vert}}

\usepackage{amssymb}
\usepackage{amsmath}
\usepackage{amsthm}
\usepackage{amsbsy}
\usepackage{psfrag}
\usepackage{pstricks}
\usepackage{graphics}
\usepackage{graphicx}
\theoremstyle{plain}
\newtheorem{theorem}{Theorem}[section]
\newtheorem{lemma}[theorem]{Lemma}

\newtheorem{example}[theorem]{Example}
\theoremstyle{remark}
\newtheorem{remark}[theorem]{Remark}

\begin{document}
\allowdisplaybreaks[4]
\numberwithin{figure}{section}
\numberwithin{table}{section}
 \numberwithin{equation}{section}
%
\title[ Discontinuous Galerkin Methods for Frictional Contact Problem with Normal Compliance]
 {Unified Analysis of Discontinuous Galerkin Methods for Frictional Contact Problem with normal compliance}

\author{Kamana Porwal}\thanks{The first author's work is supported  by CSIR Extramural Research Grant}
\address{Department of Mathematics, Indian Institute of Technology Delhi - 110016}
\email{kamana@maths.iitd.ac.in}

\author{Tanvi}\thanks{}

\address{Department of Mathematics, Indian Institute of Technology Delhi - 110016}
\email{tanviwadhawan1234@gmail.com}
\date{}
\begin{abstract}
In this article, a reliable and efficient a posteriori error estimator of residual type is derived for a class of discontinuous Galerkin methods for the frictional contact problem with reduced normal compliance which is modeled as a quasi-variational inequality. We further derive a priori error estimates in the energy norm under the minimal regularity assumption on the exact solution.  The convergence behavior of error over uniform mesh and the performance of error estimator are illustrated by the numerical results.
\end{abstract}

%
%
\maketitle
\allowdisplaybreaks
\def\R{\mathbb{R}}
\def\cA{\mathcal{A}}
\def\cK{\mathcal{K}}
\def\cN{\mathcal{N}}
\def\p{\partial}
\def\O{\Omega}
\def\bbP{\mathbb{P}}
\def\cV{\mathcal{V}}
\def\cM{\mathcal{M}}
\def\cT{\mathcal{T}}
\def\cE{\mathcal{E}}
\def\bF{\mathbb{F}}
\def\bC{\mathbb{C}}
\def\bN{\mathbb{N}}
\def\ssT{{\scriptscriptstyle T}}
\def\HT{{H^2(\O,\cT_h)}}
\def\mean#1{\left\{\hskip -5pt\left\{#1\right\}\hskip -5pt\right\}}
\def\jump#1{\left[\hskip -3.5pt\left[#1\right]\hskip -3.5pt\right]}
\def\smean#1{\{\hskip -3pt\{#1\}\hskip -3pt\}}
\def\sjump#1{[\hskip -1.5pt[#1]\hskip -1.5pt]}
\def\jumptwo{\jump{\frac{\p^2 u_h}{\p n^2}}}

\par
\noindent
\section{Introduction}
This article is devoted to the numerical analysis of the frictional contact problem with normal compliance. Frictional contact problems are of great interest since the processes involving frictional contact between two bodies occur in many engineering and industrial applications. In these problems, an elastic body, under the influence of body forces and surface tractions, comes into contact of a rigid surface on a part of its boundary (called contact boundary). The lubricated contact boundary results in a frictionless contact problem while we get frictional contact problems when the contact boundary is not lubricated.  We refer to the book by Kikuchi \& Oden \cite{KO:1988:CPBook} for modeling and detailed understanding of frictionless and frictional contact problems. In order to study these problems within the framework of  variational inequalities the first attempt was made in \cite{Duv:1976:IMP}. In most cases, the contact problems arising in real life have interface with non-zero compliance because of the presence of asperities and absorbed impurities etc in real surfaces. The frictional contact problem with normal compliance can be modeled as a quasi-variational inequality. The convergence analysis of conforming finite element approximation based on quadrilateral elements for frictional contact problem with normal compliance is studied in \cite{Oden:1993:NC}. A Cea's type error inequality of conforming finite element method for frictional contact problem with reduced normal compliance is obtained in \cite{Han:1993:FCPNC}, therein a posteriori error analysis is also discussed using regularization method.  We refer to \cite{Fern:2010:APFC} for residual type a posteriori error estimates of linear continuous finite element method for the same problem. Some more notable works on the numerical analysis of static/time dependent frictional contact problem with normal complaince can be found in \cite{Klarbring:1988:Frictional, Klarbring:1989:Frictional, Andresson:1991:quasistatic,Friction:1999:ANA,Xaio:2018:DGNC}.  
 
\par
Discontinuous Galerkin (DG) methods, which were first proposed in \cite{RH:1973:DG}, are mainly attractive due to the flexibility of using local hp adaption. The articles \cite{Arnold:1982:IPD, ABCM:2002:UnifiedDG, PAE:2012:DGbook, HW:2007:Book, Riviere:2008:DGBook} are excellent references for the comprehensive study of these methods. DG methods are also widely used to solve variational inequalities. We refer to \cite{WHC:2010:DGVI,WHC:2011:DGSP,BS:2012:LDG,GGP:2021:DGVI}  and \cite {TG:2014:VIDG,TG:2014:VIDG1,BS:2014:hp-apost,WHE:2015:ApostDG,BS:2015:hpAdapt,TG:2016:VIDG1} respectively, for a priori and a posteriori analysis of DG methods for variatonal inequalities of the first kind. The articles \cite{TG:2016:NMPDE, KPorwal:2017:Tresca} discuss the convergence analysis of DG methods over uniform mesh and adaptive mesh based on a posteriori error estimator for variational inequalities of the second kind. Further, we refer to \cite{HW:2002:FourthVI,Bostan:2004:adaptive,BH:2004:VI2,Bostan:2006:APFCP,HKS:2011:VI2,Wang:2013:VISK,MA:2015:APVI} and references therein for other works on the numerical analysis of variational inequalities of the second kind. 
In \cite{Xaio:2018:DGNC}, DG methods for frictional contact problem with normal compliance has been proposed. In this article, we first derive a residual type a posteriori error estimator of  DG methods for the frictional contact problem with reduced normal compliance which is shown to be both reliable and efficient. Followed by that, we establish an abstract a priori error estimate by assuming minimal regularity of the exact solution. The analysis is carried out in a general framework which holds for a class of DG methods. Numerical results are presented to
illustrate the theoretical findings.

\par

We consider the deformation of an elastic body unilaterally supported by a rigid foundation and occupying domain  $\O\subset \mathbb R^2$ which is a bounded polygonal domain with Lipschitz boundary $\partial \Omega=\Gamma$. The boundary $\Gamma$ is partitioned into three relatively open mutually disjoint  parts $\Gamma_D$, $\Gamma_F$ and $\Gamma_C$ with meas$(\Gamma_D)>0$. Let $S$ denotes the space of second order symmetric tensors on $\mathbb {R}^{2}$ with the scalar product defined as $\b{\b{w}:\phi}=w_{ij}\phi_{ij}~ \mbox{for}~ \b{\b{w},~\phi} \in S$ and the corresponding norm $|\b{\phi}|:=(\b{\phi}:\b{\phi})^{1/2}$.
\par
The linearized strain tensor $\b{\b{\epsilon}}$ and stress tensor $\b{\b{\sigma}}$ belong to the class of second order symmetric tensors and are defined respectively, as
\begin{align}
\b{\epsilon}(\b{u})&=\frac{1}{2} (\nabla\b{u}+\nabla\b{u}^T), \label{1.1}\\
\b{\sigma}(\b{u})&= C\b{\epsilon}(\b{u}), \label{1.2}
\end{align}
where, the vector-valued function $\b{u} : \O \subset \mathbb R^2 \rightarrow \mathbb R^2$ denotes the displacement vector and the operator $C :\O \times S \rightarrow S$ is the fourth-order elasticity tensor of the material.
In the following study, we assume elastic body to be homogeneous and isotropic, therefore
\begin{align}\label{1.3}
 C\b{\epsilon}(\b{u}):=\lambda tr(\b{\epsilon}(\b{u}))I+2\mu\b{\epsilon}(\b{u}).   
\end{align}
where, $\lambda>0$ and $\mu>0$ are Lam$\acute{e}$'s coefficients and $I$ denotes $2 \times 2$ identity matrix.
\par
\noindent
For any displacement field $\b{v}$, we adopt the notation $v_n=\b{v}\cdot \b{n}$ and $\b{v}_\tau=\b{v}-v_n\b{n}$ respectively, as its normal and tangential component on the boundary where $\b{n}$ is the outward unit normal vector to $\Gamma$. Similarly, for a tensor-valued function $\b{\sigma} : \Omega \rightarrow S$ the normal and tangential component are defined as ${\sigma}_n=\b{\b{\sigma} n}\cdot \b{n}$ and $\b{\b{\sigma}}_\tau=\b{\b{\sigma}}\b{n}-\b{\sigma}_n \b{n}$ respectively. Further, we have the following decomposition formula
\begin{align*}
  (\b{\sigma}\b{n})\cdot\b{v}= \sigma_nv_n+ \b{\sigma}_\tau\cdot\b{v_\tau}.
\end{align*}
In order to state the weak formulation for the frictional contact problem, we introduce the space $\b{V}$ of admissible displacements by
     \begin{align*}
     \b{V}&=\{\b{v} \in {[H^1(\Omega)]}^2: \b{v}~=~\b{0}~on~\Gamma_D\}.
     \end{align*}
\noindent
Given $\b{f} \in [L^2(\O)]^2,~\b{g} \in [L^2(\Gamma_F)]^2,~g_a \in H^{1/2}(\Gamma_C)$ with $g_a>0$, variational formulation of 
the frictional contact problem with normal compliance is to find $\b{u}~\in~\b{V}$ s.t.
    \begin{align}\label{1.9}
    a(\b{u, v - u})+j_n(\b{u, v - u})+j_\tau(\b{u}, \b{v})-j_\tau(\b{u}, \b{u})\geq (\b{f, v - u})~~\forall ~~\b{v} \in \b{V},
    \end{align}
where, the bilinear form $a(\cdot, \cdot)$, the functional $j_n(\cdot, \cdot)$, $j_\tau(\cdot, \cdot)$ and the linear functional $(\b{f, \cdot})$ are defined by
    \begin{align*}
    a(\b{w, v})&=\int_\Omega \b{\sigma} (\b{w})\colon \b{\epsilon}(\b{v})~dx,\\
    j_n(\b{w, v})&=\int_{\Gamma_C} c_{n}(w_{n}-g_a)_+^{m_n}{v_n}~ds,\\
    j_\tau(\b{w, v})&=\int_{\Gamma_C} c_{\tau}(w_n-g_a)^{m_t}_+|\b{{v_\tau}}|~ds,\\
    (\b{f, v})&= \int_\Omega \b{f} \cdot \b{v}~dx + \int_ {\Gamma_F} \b{g} \cdot \b{v}~ds~~\forall~ \b{w},\b{v}~\in~\b{V},
    \end{align*}
with $c_n,c_{\tau} \in L^{\infty}(\Gamma_C)$, $1 \leq m_n < \infty$ and $0 \leq m_t < \infty$ .
The classical(strong) form associated to the quasi variational inequality \eqref{1.9} is to find the displacement vector $\b{u}: \Omega \rightarrow \mathbb{R}^2$ satisfying the equations \eqref{1.4}-\eqref{1.8},
\begin{align}
    \b{-div} ~~\b{\sigma}(\b{u}) &= \b{f} ~~~~\textit{in}~\Omega,\label{1.4}\\
    \b{u} &= \b{0} ~~~~\textit{on}~\Gamma_D, \label{1.5}\\
    \b{\sigma}(\b{u})\b{n} &= \b{g} ~~~~\textit{on}~\Gamma_F,\label{1.6}\\
    {\sigma}_{n}(\b{u}) &= -c_{n}(u_{n}-g_a)_+^{m_n}~~~{on} ~~~\Gamma_C, \label{1.7}
    \end{align}
    \begin{align} \label{1.8}
    \begin{split}
    \left.\begin{aligned}
    |\b{\b{\sigma}_\tau}| &< c_{\tau}(u_{n}-g_a)_+^{m_t} \implies {\b{u_{\tau}}} = 0 \\
    |\b{\b{\sigma}_\tau}| &= c_{\tau}(u_{n}-g_a)_+^{m_t} \implies {\b{u_{\tau}}} = -\lambda \b{\sigma}_\tau~\textit{for some}~\lambda \geq 0
    \end{aligned}\right\}  on~~\Gamma_C.
    \end{split}
    \end{align}

 The equation ${(1.5)}$ is the equilibrium equation, in which volume forces of density $\b{f}$ acts in $\O$. The equation ${(1.6)}$ justifies that displacement field vanishes on $\Gamma_D$, which  means that the body is clamped on $\Gamma_D$. Surface traction of density $\b{g}$ acts on $\Gamma_F$ in ${(1.7)}$. The normal compliance condition is given by ${(1.8)}$ where $g_a$ is the initial gap between the body and foundation, $u_n$ is the normal displacement and $(u_n-g_a)_+$ represents the penetration of the body in the foundation. Here, ${c_n} \in  L^{\infty}(\Gamma_C)$ is a non negative function  with the property $c_n(\b{x})=0$ for $\b{x}\leq 0$. The relation ${(1.9)}$ form a version of the Coulomb's Law of dry friction where ${c_\tau} \in L^{\infty}(\Gamma_C)$ is a non negative friction bound  with the property $c_\tau(\b{x})=0$ for $\b{x}\leq 0$.\\

In this article, we will analyze the frictional contact problem with reduced normal compliance law \cite{Han:1993:FCPNC} i.e. $m_t=0$. Therefore, \eqref{1.8} steps down to
\begin{align*}
   \begin{split}
    \left.\begin{aligned}
    |\b{\b{\sigma}_\tau}| &< c_{\tau} \implies {\b{u_{\tau}}} = 0 \\
    |\b{\b{\sigma}_\tau}| &= c_{\tau} \implies {\b{u_{\tau}}} = -\lambda \b{\sigma}_\tau~\textit{for some}~\lambda \geq 0
    \end{aligned}\right\}  on~~\Gamma_C.
    \end{split}
\end{align*}
In this case the functional $j_\tau(\b{u},\b{v})$ reduces to $j_\tau(\b{v})$ which is defined by 
\begin{align*}
j_\tau(\b{v})=\int_{\Gamma_C} c_{\tau}|\b{{v_\tau}}|~ds.
\end{align*}
The variational formulation \eqref{1.9} reduces to the following problem: to find the displacement vector $\b{u} \in \b{V}$ s.t.
    \begin{align}\label{1.10}
    a(\b{u, v - u})+j_n(\b{u, v - u})+j_\tau(\b{v})-j_\tau(\b{u})\geq (\b{f, v - u})~~\forall ~~\b{v} \in \b{V}.
    \end{align}
The existence and uniqueness of the solution $\b{u}$ of the problem $\eqref{1.10}$ follows from \cite{Han:1993:FCPNC}.
\par
\noindent
 We define,
    \begin{align*}
    \Lambda= \{ \b{\mu} \in [L^{\infty}(\Gamma_C)]^2 ~:~ |\b{\mu}|\leq 1 ~ \text{a.e.~ on} ~~ \Gamma_C\}.
\end{align*}
Now, we will characterize the continuous solution $\b{u}$ of $\eqref{1.10}$ through the use of Lagrange multiplier \cite{HW:2002:FourthVI,KPorwal:2017:Tresca}.
\begin{lemma}\label{ref:Cchar}
There exists $\b{\lambda_{\tau} }\in \Lambda$ such that  
    \begin{align*}
    a(\b{u,v})+j_n(\b{u,v})+g(\b{\lambda_\tau,v})&=(\b{f},\b{v}) \quad ~\forall~ \b{v} \in \b{V}, \\
    \b{\lambda_{\tau}}\cdot\b{u_{\tau}}&=|\b{u_{\tau}}| ~~a.e~~ on~ \Gamma_C,
    \end{align*}
where
    \begin{align*}
    g(\b{\lambda_{\tau},v})=\int_{\Gamma_C}c_\tau \b{\lambda_{\tau}\cdot v_\tau}~ds.
    \end{align*}
\end{lemma}
In the subsequent analysis, we also require the following bound on the exact solution $\b{u}$ of \eqref{1.10} by load vectors \cite{Han:1993:FCPNC}.
\begin{lemma}\label{lem:ubound} 
Let $\b{u} \in \b{V}$ be the solution of continuous problem \eqref{1.10}. Then 
       \begin{align*}
        \|\b{u}\|_{H^1 (\Omega)}\leq C (\|\b{f}\|_{L^2(\Omega)}+\|\b{g}\|_{L^2(\Gamma_F)})  
       \end{align*}
where $C$ is a constant independent of $h$.
\end{lemma}
\par
In view of the following imbedding result \cite{Ciarlet:1978:FEM},
\begin{align}\label{im}
   H^1(\O) \hookrightarrow L^q(\Gamma_C) ~\forall~ q\in [1,\infty),
\end{align} 
  it can be be observed that $\sigma_n(\b{u}) \in L^2(\Gamma_C)$.
  \par\noindent
The rest of the article is arranged as follows: In next section, we introduce notations and present some useful preliminary results which will be used in subsequent analysis. DG formulation is presented for the continuous problem \eqref{1.10} in Section \ref{sec:DP}. Followed by that in Section \ref{sec:Apost}, a posteriori error analysis of DG methods for the frictional contact problem with reduced normal compliance \eqref{1.10} has been established.  A priori error analysis  with minimal regularity on exact solution $\b{u}$ of \eqref{1.10} is carried out  in Section \ref{sec:Apriori}. In Section \ref{sec:NumR},  numerical results are presented to illustrate the theoretical findings. Finally, we present the conclusions of this article in Section \ref{sec:Summ}.
\section{Preliminaries} \label{sec:Prelims}
\par
\subsection{Notations } The following notations will be used in the further analysis.
     \begin{align*}
     \mathcal{T}_h &:= \text{a family of regular triangulation of $\O$},\\
     \mathcal{E}_h &:= \text{set of all edges of } \mathcal{T}_h,\\
     \mathcal{E}_h^i &:= \text{set of all interior edges of $\mathcal{T}_h$},\\
     \mathcal{E}_h^b &:= \text{set of all boundary edges of $\mathcal{T}_h$},\\
     \mathcal{E}_h^D &:= \{ e \in \mathcal{E}_h^b : e \subset \Gamma_D\},\\
     \mathcal{E}_h^F &:=\{ e \in \mathcal{E}_h^b : e \subset \Gamma_F\}, \\
     \mathcal{E}_h^C &:= \{ e \in \mathcal{E}_h^b : e \subset \Gamma_C\},\\
     \mathcal{E}_h^o &:= \mathcal{E}_h^i \cup \mathcal{E}_h^D,\\
     \mathcal{T}_p &:= \text{set of all elements of $\mathcal{T}_h$ sharing the vertex $p$},\\
     \mathcal{T}_e &:= \text{set of all elements of $\mathcal{T}_h$ sharing the edge $e$},\\
     \mathcal{V}_h^i&:= \text{set of all interior vertices of }\mathcal{T}_h,\\
     \mathcal{V}_{T}&:= \text{set of all vertices of element $T$},\\
     \mathcal{V}_{\partial\Omega}&:= \text{set of all boundary vertices of }\mathcal{T}_h,\\
     \mathcal{V}_h ^F&:= \text{set of all  vertices of }\mathcal{T}_h  \text{ lying on }\Gamma_F,\\
     \mathcal{V}_h ^C&:= \text{set of all vertices of }\mathcal{T}_h  \text{ lying on }\overline{\Gamma_C},\\
     \mathcal{V}_h ^D&:= \text{set of all vertices of }\mathcal{T}_h \text{ lying on~}  \overline{ \Gamma_D},\\
     T &:= \text{an element of }\mathcal{T}_h, \\
     h_T &:= \text {diameter of $T$},\\
     h &:= \text{max}\{h_T : T \in \mathcal{T}_h\},\\
     h_e &:= \text{length of an edge $e$},\\
     P_{k}(T)&:= \text{space of polynomials of degree}\leq k ~\text{defined on ~} T, ~0 \leq k \in \mathbb{Z}.\\
     \end{align*}
     The notations, $\nabla_h(\b{v})$ and $\b{div_h(\b{v})}$, respectively denote elementwise gradient and divergence i.e. for $T \in \mathcal{T}_h,$~ $\nabla_h(\b{v})|_{T}  = \nabla \b{v}, ~\b{div_h(\b{v})}|_T= \b{div(\b{v})}$. Further, for $\b{v} \in \mathcal{\b{V_h}}$, $\b{\epsilon_h}(\b{v})$ and $\b{\sigma_h}(\b{v})$ are such that $\b{\epsilon_h}(\b{v})|_T =\b{\epsilon(\b{v})},~ T \in \mathcal{T}_h$ and $\b{\sigma_h}(\b{v})=2\mu \b{\epsilon_h}(\b{v})+ \lambda tr(\b{\epsilon_h}(\b{v}))I.$
     \par 
     \noindent
In order to deal with nonsmooth functions, we define the broken Sobolev space $[H^1(\Omega,\mathcal{T}_h)]^2$ as
\begin{align*}
    [H^1(\Omega,\mathcal{T}_h)]^2 := \{ \b{v}\in [L^2(\Omega)]^2 : \b{v}|_T\in [H^1(T)]^2~ \forall~ T\in \cT_h\}
\end{align*}
and the corresponding norm on this space is defined as $\|\cdot\|^2_{1,h}=\sum_{T \in \mathcal{T}_h} \|\cdot\|^2_{H^1(T)}$.
\par
Let $e \in \cE_h^i$ be an interior edge and let $T^{+}$ and $T^{−}$ be the neighbouring elements s.t. $e \in \partial{T}^{+} \cup \partial T^{−}$ and let $\b{n}^{\pm}$ is the unit outward normal vector on $e$ pointing from $T^{+}$ to $T^{-}$ s.t. $\b{n^{-}}= - \b{n^{+}}.$ For a vector valued function $\b{v}\in [H^1(\Omega,\mathcal{T}_h)]^2$ and a matrix valued function $\b{\phi}\in [H^1(\Omega,\mathcal{T}_h)]^{2\times 2}$, averages $\smean{\cdot}$ and jumps $\sjump{\cdot}$ across the edge $e$ are defined as follows:     
\begin{align*}
&\smean{\b{v}}= \frac{1}{2} (\b{v^+}+\b{v^-})~~\text{and}~~\sjump{\b{v}}=\frac{1}{2} (\b{v^+}\otimes \b{n^+}+\b{n^+}\otimes\b{v^+}+\b{v^-}\otimes\b{n^-}+\b{n^-}\otimes\b{v^-}),\\
&\smean{\b{\phi}}= \frac{1}{2} (\b{\phi^+}+\b{\phi^-})~ \text{and}~~\sjump{\b{\phi}}=\b{\phi^+}\b{n^+}+ \b{\phi^-}\b{n^-},
\end{align*}
where
    $\b{v^{\pm}}=\b{v}|_{{T}^\pm},~\b{\phi^{\pm}}=\b{\phi}|_{{T}^\pm}.$
\par
For any $e \in \cE_h^b$, it is clear that there is a triangle $T \in \mathcal{T}_h$ such that $e\in \partial T \cap \partial\O$. Let $\b{n_e}$ be the unit normal of $e$ that points outside $T$. Then, the averages $\smean{\cdot}$ and jumps $\sjump{\cdot}$ of vector valued function $\b{v}\in [H^1(\Omega,\mathcal{T}_h)]^2$ and a matrix valued function $\b{\phi}\in [H^1(\Omega,\mathcal{T}_h)]^{2\times 2}$ are defined as follows:
\begin{align*}
&\smean{\b{v}}=\b{v},~~\text{and}~~\sjump{\b{v}}=\frac{1}{2}(\b{v}\otimes \b{n_e}+\b{n_e}\otimes\b{v}),
\\
&\smean{\b{\phi}}= \b{\phi},~~\text{and}~~\sjump{\b{\phi}}=\b{\phi}\b{n_e}.
\end{align*}
In the above definitions $\b{v}\otimes \b{n}$ is a $2\times 2$ matrix with $v_in_j$ as its $(i,j)^{th}$ entry.
\par
The discontinuous finite element space $\b{V_h}$ is defined as 
\begin{align*}
\b{V_h}&=\{ \b{v} \in [L^2(\Omega)]^2: \b{v}|_T \in [P_1(T)]^2~~ \forall~ T \in \mathcal{T}_h\}.
\end{align*}
 In the subsequent analysis, we will also require the conforming finite element subspace defined by $\b{V_c} =\b{V_h}\cap \b{V}$, which we choose as standard Lagrange linear finite element space.
\par
Throughout the article, $C$ denotes a generic positive constant that is independent of mesh parameter $h$. The notation  $X \sim Y$ says that there exists positive constants $C_1, C_2$ such that $C_1 Y \leq X \leq C_2 Y.$
\par\noindent
The following Clement type approximation result \cite{BScott:2008:FEM} will be useful in establishing convergence analysis.
\begin{lemma} \label{lem:Interpolation}
 Let $\b{v}\in \b{V}$. Then there exist $\b{v_h}\in \b{V_c}$ such that on any $ T \in \mathcal{T}_h$,
    \begin{align*}
    \|\b{v}-\b{v_h}\|_{H^s(T)} \leq Ch_T^{1-s}\|\b{v}\|_{H^1({\mathcal{T}_T})},  ~~~ s=0,1,
    \end{align*}
    where $\mathcal{T}_T=\{ T' \in \mathcal{T}_h : \overline{T'}\cap \overline{T} \neq \phi\}$ and $C$ is a positive constant independent  of $h$ .
\end{lemma}

The following inverse and trace inequalities \cite{BScott:2008:FEM,PAE:2012:DGbook} will also be frequently used in the subsequent analysis.
\begin{lemma}(Discrete trace inequality)
Let $\b{v} \in [H^1(T)]^2$ for $T \in \cT_h$ and $e$ be an edge of $T$. Then, it holds that
\begin{eqnarray}\label{2.1}
&\lVert \b{v} \rVert_{L^2{(e)}}\leq C\big(h_e^{-1}\lVert \b{v} \rVert_{L^2(T)}^2+h_e \lVert \nabla \b{v} \rVert_{L^2(T)}^2\big)^{\frac{1}{2}},
\end{eqnarray}
where $C$ is a constant independent of h.
\end{lemma}
\begin{lemma} (Inverse inequalities)
Let $T \in \cT_h$ and e be an edge of T. Then, it holds that for any $\b{v} \in \b{V_h}$
\begin{eqnarray}
\lVert \b{v} \rVert_{L^\infty{(e)}} \leq C h_e^{-\frac{1}{2}} \lVert \b{v} \rVert_{L^2(e)},\label{2.2}\\
\lVert \b{v} \rVert_{L^2{(e)}} \leq C h_e^{-\frac{1}{2}} \lVert \b{v} \rVert_{L^2(T)} & \forall~ T\in \mathcal{T}_h, \label{2.3}\\
\lVert \nabla \b{v} \rVert_{L^2{(T)}} \leq C h_T^{-1} \lVert \b{v} \rVert_{L^2(T)} & \forall~ T\in \mathcal{T}_h, \label{2.4}
\end{eqnarray}
where C is a constant independent of h.
\end{lemma}

\subsection{Enriching Operator}
An enriching map $E_h:\b{V_h}\rightarrow
\b{V_c}$ plays a crucial role in deriving a posteriori error estimates for the class of discontinuous Galerkin
methods as it maps non-conforming function to conforming function \cite{Brenner:1996:SchwarzNC, Brenner:1999:multigridNC,
Brenner:2001:Poincare,
Brenner::2004:Korn}.
\par
\noindent As we know, that any function in $\b{V_c}$ is uniquely
determined by the nodal values at the vertices $\cV_h$ of $\cT_h$, therefore, 
for $\b{v_h}\in \b{V_h}$, we define $E_h\b{v_h}\in \b{V_c}$ by averaging as follows:
\smallbreak
$E_h \b{v_h}(p)$  $=$
$\begin{cases}
\dfrac{1}{|\mathcal{T}_p|}\sum_{T \in \mathcal{T}_p } \b{v_h}|_T(p)  & \text{for }p \in \mathcal{V}_h ^F \cup \mathcal{V}_h ^i \cup \mathcal{V}_h ^C, \\ 
  0& \text{for }p \in \mathcal{V}_h ^D.  
\end{cases}$\\
where $\lvert \mathcal{T}_p \rvert$ denotes the cardinality of $\mathcal{T}_p$.
\par
In the following lemma, we state the approximation properties of smoothing map $E_h$ \cite{TG:2016:VIDG1, KPorwal:2017:Tresca}.
\begin{lemma} \label{lem:EO}  It holds that
    \begin{align*}
    \sum_{T \in \mathcal{T}_h}\Big(h_T^{-2}\| E_h\b{v}-\b{v}\|_{L^2(T)}^2 + \| \nabla({E_h\b{v}-\b{v}})\|_{L_2(T)}^2 \Big)\leq C \Bigg( \sum_{e \in \mathcal{E}_h}\dfrac{1}{h_e}\|\sjump{\b{v}}\|_{0,e}^2\bigg) \quad \forall \b{v}\in \b{V_h}.
    \end{align*}

\end{lemma}
\section{Discrete Problem} \label{sec:DP}
\subsection{DG Formulations}
In the following subsection, we present DG formulations for solving the quasi-variational inequality \eqref{1.10}. 
In \cite{Xaio:2018:DGNC} several DG methods have been considered for the frictional problem with normal compliance for which the bilinear form $B_h(\cdot, \cdot)$ are listed below. Let $r_0$ and $r_e$ denote the global and local lifting operators, respectively \cite{ABCM:2002:UnifiedDG,WHC:2011:DGSP}. Further, in defining the bilinear forms, we use the shorter notations $(\b{w},\b{v})_{\O},~ \langle \b{w},\b{v}\rangle_{\cE^o_h}$ and $g$ instead of $\int_\O \b{wv}~dx,~\int_{\cE^o_h }\b{wv}~ds$ and $~ \int_\O\b{\sigma}_h(\b{u_h}):\b{\epsilon}_h(\b{v_h})~dx$ respectively. 
\smallbreak
1. \textbf{SIPG method} \cite{Xaio:2018:DGNC, WHC:2011:DGSP, Arnold:1982:IPD}:
\begin{align*}
B_{h}^{(1)}(\b{u_h},\b{v_h}) &= g - \langle \sjump{\b{u_h}},\smean{\b{\sigma}_h(\b{v_h})}\rangle -  \langle \sjump{\b{v_h}},\smean{\b{\sigma}_h(\b{u_h})}\rangle + \int_{\cE_h^o}\eta h_e^{-1}\sjump{\b{u_h}}:\sjump{\b{v_h}}~ds,
\end{align*}
for $\b{u_h}, \b{v_h} \in \b{V_h}$ and $\eta \geq \eta_o>0.$
\bigbreak
2. \textbf{NIPG method} \cite{WHC:2011:DGSP, Xaio:2018:DGNC}:
\begin{align*}
B_{h}^{(2)}(\b{u_h},\b{v_h}) &= g + \langle \sjump{\b{u_h}},\smean{\b{\sigma}_h(\b{v_h})}\rangle -  \langle \sjump{\b{v_h}},\smean{\b{\sigma}_h(\b{u_h})}\rangle + \int_{\cE_h^o}\eta h_e^{-1}\sjump{\b{u_h}}:\sjump{\b{v_h}}~ds,
\end{align*} 
for $\b{u_h}, \b{v_h} \in \b{V_h}$ and $\eta > 0.$
\bigbreak
3.  \textbf{Bassi et al.} \cite{WHC:2011:DGSP, Xaio:2018:DGNC}:
\begin{align*}
B_{h}^{(3)}(\b{u_h},\b{v_h}) &= g - \langle \sjump{\b{u_h}},\smean{\b{\sigma}_h(\b{v_h})}\rangle -  \langle \sjump{\b{v_h}},\smean{\b{\sigma}_h(\b{u_h})}\rangle+ \sum_{e \in \cE^o_h}\int_\O \eta C\b{r}_e(\sjump{\b{u_h}}):\b{r}_e(\sjump{\b{v_h}})~dx,
\end{align*} 
for $\b{u_h}, \b{v_h} \in \b{V_h}$ and $\eta > 3.$
\bigbreak
4.     \textbf{Brezzi et al.} \cite{ABCM:2002:UnifiedDG, WHC:2011:DGSP, BMMPR:2000:BrezziDG}:
\begin{align*}
B_{h}^{(4)}(\b{u_h},\b{v_h}) &= g - \langle \sjump{\b{u_h}},\smean{\b{\sigma}_h(\b{v_h})}\rangle -  \langle \sjump{\b{v_h}},\smean{\b{\sigma}_h(\b{u_h})}\rangle + (C{r}_0(\b{\sjump{\b{u_h}})},r_0(\sjump{\b{v_h}}))\\&+\sum_{e \in \cE^o_h}\int_\O \eta C\b{r}_e(\sjump{\b{u_h}}):\b{r}_e(\sjump{\b{v_h}})~dx,
\end{align*} 
for $\b{u_h}, \b{v_h} \in \b{V_h}$ and $\eta > 0.$
\bigbreak
5. \textbf{LDG Method}  \cite{CCPS:2000:LDG,Chen::2010:ldg}:
\begin{align*}
B_{h}^{(5)}(\b{u_h},\b{v_h}) &= g - \langle \sjump{\b{u_h}},\smean{\b{\sigma}_h(\b{v_h})}\rangle -  \langle \sjump{\b{v_h}},\smean{\b{\sigma}_h(\b{u_h})}\rangle + (C{r}_0(\b{\sjump{\b{u_h}})},r_0(\sjump{\b{v_h}}))\\&+\int_{\cE_h^o}\eta h_e^{-1}\sjump{\b{u_h}}:\sjump{\b{v_h}}~ds,
\end{align*}   
for $\b{u_h}, \b{v_h} \in \b{V_h}$ and $\eta > 0.$
\par
Let $B_h(\cdot, \cdot)$ represents one of the five bilinear form $B^{(i)}_h(\cdot, \cdot),~ 1\leq i\leq 5$. Then, the corresponding discrete formulation of the model problem $\eqref{1.10}$ is to find $\b{u_h} \in \b{V_h}$ such that 
    \begin{align}\label{4.1}
    B_h(\b{u_h,v_h-u_h})+j_n(\b{u_h,v_h-u_h})+j_\tau(\b{v_h})-j_\tau(\b{u_h})\geq (\b{f,v_h-u_h}) \quad \forall~ \b{v_h} \in \b{V_h}.
    \end{align}
where we rewrite the bilinear form $B_h(\cdot,\cdot)$ as
    \begin{align*}
     B_h(\b{u_h,v_h})= a_h(\b{u_h,v_h}) +b_h(\b{u_h,v_h}),
     \end{align*}
where
    \begin{align*}
    a_h(\b{u_h,v_h})=\int _\Omega \b{\sigma}_h(\b{u_h})\colon \b{\epsilon}_h(\b{v_h})~dx
    \end{align*}
and bilinear form $b_h(\cdot,\cdot)$ consists of all the remaining terms that accounts for consistency and stability. 
A key observation is that the bilinear form $b_h(\cdot,\cdot)$ for all the DG methods (1) - (5) satisfies the following estimate:
\begin{align}\label{4.2}
\lvert b_h(\b{w},\b{v})\rvert \leq C \bigg(\sum_{e \in \mathcal{E}_h^o}\int_{e}\frac{1}{h_e}\sjump{\b{w}}^2~ds\bigg)^{1/2}\lvert\b{v}\rvert_{H^1(\Omega)}~\forall \b{w} \in \b{V_h},~ \b{v} \in \b{V_c}.
\end{align}
    Define norm $\norm{ \cdot }_h$ on the space $\b{V_h}$ as 
    \begin{align*}
    \norm{\b{v}}_h^2=\mid \b{v}\mid^2_h+\mid \b{v}\mid_*^2,
    \end{align*}
where
    \begin{align*}
    \mid \b{v} \mid^2_h = \sum_{T\in\mathcal{T}_h} \mid \b{v}\mid_T^2~,
~~~~~~~~
    \mid \b{v} \mid^2_* = \sum_{e\in\mathcal{E}_h^0} h_e^{-1}\| \sjump{\b{v}}\|_{0,e}^2
    \end{align*}
with
    \begin{align*}
    \mid \b{v} \mid^2_T = \int_{T} C\b{\epsilon}(\b{v}):\b{\epsilon}(\b{v})~dx,~~~~
    \| \sjump{\b{v}}\|_{0,e}^2= \int_e\sjump{\b{v}}:\sjump{\b{v}}~ds.
\end{align*}
Note, the norm $\norm{\cdot}_h$ is equivalent to usual DG norm
$\|\b{v}\|^2_{1,h}+\mid \b{v}\mid^2_*$ by Korn's inequality and Poincar$\acute{e}$ Fredrichs inequality for piece wise $H^1$ spaces \cite{Brenner:2001:Poincare, Brenner::2004:Korn}.
\par
\noindent
The existence and uniqueness of the discrete problem \eqref{4.1} is discussed in \cite{Xaio:2018:DGNC}.
Analogous to the continuous problem, following is the characterization of the discrete problem \eqref{4.1}.
\begin{lemma}\label{ref:Dchar}
There exists a unique Lagrange multiplier $\b{\lambda_{h\tau}} \in \Lambda $ such that the solution $\b{u_h}$ of the discrete problem \eqref{4.1} can be characterized by
     \begin{align}
     B_h(\b{u_h, v_h})+j_n(\b{u_h, v_h})+g(\b{\lambda_{h\tau}, v_{h}})=&\b{(f,v_h})~~ \quad \forall \b{v_h} \in \b{V_h}, \label{4.3}\\
     \b{\lambda_{h\tau}}\cdot\b{u_{h\tau}}=&|\b{u_{h\tau}}| ~~a.e~~ on~ \Gamma_C. \label{4.4}
     \end{align}
\end{lemma}
Since $j_n(\b{u_h,v_h})$ is linear in the second component, henceforth the proof of the last lemma follows using the similar arguments as in Lemma 3.1 of \cite{KPorwal:2017:Tresca}.
\par\noindent
As in the case of continuous solution $\b{u}$ of $\eqref{1.10}$, the discrete solution $\b{u_h}$ of $\eqref{4.1}$ is also uniformly bounded by load vectors.
\begin{lemma}\label{lem:uhbound}  Let $\b{u_h} \in \b{V_h}$ be the solution of the discrete  problem. Then
       \begin{align*}
      \|\b{u_h}\|_{1,h}\leq C (\|\b{f}\|_{L^2(\Omega)}+\|\b{g}\|_{L^2(\Gamma_F)})  
       \end{align*}
where $C$ is a constant independent of $h$.
\end{lemma}
This lemma can be proved on the same lines as in Theorem 2.3 of \cite{Han:1993:FCPNC}.
\section{A posteriori error analysis} \label{sec:Apost}
In this section, we derive a residual-type estimator for the error $\norm{\b{u-u_h}}_h$ and study a posteriori error analysis. The error estimators are defined by
    \begin{align*}
    \eta_1^2 =& \sum_{T \in \mathcal{T}_h} h_T^2\|\b{f}\|^2_{L^2(T)},\\
    \eta_2^2 =& \sum_{e \in \mathcal{E}_h^i} h_e\|\sjump{\b{\sigma}_h(\b{u_h})}~ \|^2_{L^2(e)},\\
    \eta_3^2 =& \sum_{e \in \mathcal{E}_h^0} \frac{\eta}{h_e}\|\sjump{\b{u_h}}\|^2_{L^2(e)},\\
    \eta_4^2 =& \sum_{e \in \mathcal{E}_h^C} h_e\|\b{\sigma}_{h\tau}(\b{u_h})+  c_{\tau}\b{\lambda_{h\tau}}\|^2_{L^2(e)},\\
    \eta_5^2 =& \sum_{e \in \mathcal{E}_h^F} h_e\|\b{\sigma}_h(\b{u_h})\b{n}-\b{g}\|^2_{L^2(e)},\\
    \eta_6^2=& \sum_{e \in \mathcal{E}_h^C} h_e\|\b{\sigma}_{hn}(\b{u_h})+c_n(u_{hn}-g_a)^{m_n}_+\|^2_{L^2(e)}.
    \end{align*}
The total residual estimator $\eta_h$ is defined by 
    \begin{align*}
    \eta_h^2=\eta_1^2+\eta_2^2+\eta_3^2+\eta_4^2+\eta_5^2+\eta_6^2.
    \end{align*}
We will use the following integration by parts formula in the subsequent analysis:
\begin{align*}
   \int_{\Omega} \b{\sigma_h}(\b{w}):\b{\epsilon_h}(\b{v})~dx 
&=-\int_{\Omega}\b{div_h}~\b{\sigma}_h(\b{w})\cdot \b{v}~dx + \sum_{e\in \mathcal{E}_h^i}\int_{e}\sjump{\b{\sigma}_h(\b{w})}\cdot\smean{\b{v}}~ds +  \sum_{~e\in \mathcal{E}_h \ }\int_{e}\smean{\b{\sigma}_h(\b{w})}: \sjump{\b{v}}~ds
\end{align*}
for all $\b{v}$, $\b{w}$ $\in$ $[H^1(\Omega, \mathcal{T}_h)]^2$.
\par
Next, we establish the reliability of the error estimator $\eta_h$.
\subsection{Reliability Estimates}
In the following subsection, we derive the  upper bound for the discretization error by error estimator $\eta_h$. 
\begin{theorem} Let $\b{u}$ and $\b{u_h}$ be the solution of \eqref{1.10} and \eqref{4.1}, respectively. Then, there exist a positive constant $C$ independent of $h$ s.t. 
$$\norm{\b{u-u_h}}^2_h +\sum_{e\in \mathcal{E}_h^C}h_e\|{\sigma}_n(\b{u_h-u})\|^2_{L^2(e)}\leq C\bigg(\eta_h^2 + \sum_{e \in \mathcal{E}_h^C}h_e\|c_\tau\|^2_{L^2(e)}\bigg).$$
\end{theorem}
\begin{proof} We have,
    $$\norm{\b{u-u_h}}^2_h \leq\sum _{T \in \mathcal{T}_h}|\b{u-u_h}|^2_{1,T}+ \eta_3^2.$$
 Using Lemma \ref{lem:EO}, we note that
    \begin{align*}
    \sum _{T \in \mathcal{T}_h}|\b{u-u_h}|^2_{1,T}&\leq  \sum _{T \in \mathcal{T}_h}|\b{u}-E_h\b{u_h}|^2_{1,T}+\sum _{T \in \mathcal{T}_h}|E_h\b{u_h}-\b{u_h}|^2_{1,T} \\
    &\leq  \sum _{T \in \mathcal{T}_h}|\b{u}-E_h\b{u_h}|^2_{1,T}+\eta_3^2.
    \end{align*}
Set $\b{\phi}=\b{u}-E_h\b{u_h} \in \b{V}$.  Lemma \ref{lem:Interpolation} guarantees the approximation of $\b{\phi}$ as $\b{\phi_h }\in \b{V_c}$. Using the $\b{V}$-ellipticity of the bilinear form $a(\cdot,\cdot)$, characterization in terms of multipliers for the continuous and discrete solutions stated in Lemma \ref{ref:Cchar} and Lemma \ref{ref:Dchar}, we obtain
\begin{equation*}
   \begin{aligned}
   \sum_{T \in\cT_h}|\b{u}-E_h\b{u_h}|^2_{1,T}&\leq a(\b{u}-E_h\b{u_h},\b{\phi})\\
   &=(\b{f,\phi})-j_n(\b{u,\phi})-g(\b{\lambda_\tau,\phi})-a(E_h\b{u_h},\b{\phi})\\
   &=(\b{f,\phi-\phi_h})+(\b{f,\phi_h})-j_n(\b{u,\phi})-g(\b{\lambda_\tau,\phi})+a_h(\b{u_h}-E_h\b{u_h},\b{\phi})\\&-a_h(\b{u_h,\phi-\phi_h})-a_h(\b{u_h,\phi_h})\\
   &=(\b{f,\phi-\phi_h})-j_n(\b{u,\phi})-g(\b{\lambda_\tau,\phi})+a_h(\b{u_h}-E_h\b{u_h},\b{\phi})\\&+ b_h(\b{u_h},\b{\phi_h})+j_n(\b{u_h,\phi_h})+g(\b{\lambda_{h\tau},\phi_h})-a_h(\b{u_h,\phi-\phi_h})\\
   &=(\b{f,\phi-\phi_h})-g(\b{\lambda_{h\tau},\phi-\phi_h})-j_n(\b{u_h,\phi-\phi_h})+b_h(\b{u_h},\b{\phi_h})\\&-a_h(\b{u_h,\phi-\phi_h})+j_n(\b{u_h,\phi})-j_n(\b{u,\phi})+g(\b{\lambda_{h\tau},\phi})-g(\b{\lambda_\tau,\phi})\\&+a_h(\b{u_h}-E_h\b{u_h},\b{\phi})\\
   &=T_1+T_2+T_3+T_4,
 \end{aligned}
 \end{equation*}
where
    \begin{align*}
    T_1&= (\b{f,\phi-\phi_h})-g(\b{\lambda_{h\tau},\phi-\phi_h})-a_h(\b{u_h,\phi-\phi_h})+b_h(\b{u_h,\phi_h})-j_n(\b{u_h,\phi-\phi_h}),\\
    T_2&= a_h(\b{u_h}-E_h\b{u_h},\b{\phi}),\\
    T_3&=g(\b{\lambda_{h\tau},\phi})-g(\b{\lambda_{\tau},\phi}),\\
    T_4&=j_n(\b{u_h,\phi})-j_n(\b{u,\phi}).
    \end{align*}
We now estimate $T_i,~1\leq i \leq 4$ individually. Using integration by parts in the third term of $T_1$ and gathering all the terms, we find 
 
     \begin{align*}
     T_1&= \sum_{T\in \mathcal{T}_h} \int_{T} \b{f\cdot(\phi-\phi_h})~dx
     +\sum_{e \in \mathcal{E}_h^F}\int_{e}(\b{g}-\b{\sigma}_h(\b{u_h})\b{n_e})\cdot (\b{\phi-\phi_h})~ds \nonumber  \nonumber\\ &-\sum_{e\in \mathcal{E}_h^C}\int_{e}({\sigma}_{hn}(\b{u_{h}})+ c_{n}(u_{hn}-g_a)_+^{m_n}) \cdot (\phi-\phi_h)_n~ds-\sum_{e\in \mathcal{E}_h^i}\int_{e}\sjump{\b{\sigma}_h(\b{u_h})}\cdot
     \smean{\b{\phi-\phi_h}}~ds \\&+ b_h(\b{u_h}, \b{\phi_h}) -\sum_{e \in \mathcal{E}_h^C} \int _e (\b{\sigma}_{h\tau}(\b{u_h})+c_\tau \b{\lambda_{h\tau})\cdot (\phi-\phi_h)_\tau}~ds. 
     \end{align*}
Now, we  evaluate the terms on right hand side in the last equation one by one. The first term is bounded by using Cauchy-Schwartz inequality and Lemma \ref{lem:Interpolation} as follows:
    \begin{align*} 
    \sum_{T\in \mathcal{T}_h} \int_{T} \b{f\cdot(\phi-\phi_h)}~dx
    &\leq\bigg(\sum_{T\in \mathcal{T}_h}h_T^2\| \b{f}\|^2_{L^2(T)}\bigg)^{1/2} \bigg(\sum_{T\in \mathcal{T}_h}h_T^{-2}\|\b{\phi-\phi_h}\|^2_{L^2(T)}\bigg)^{1/2}\\
    &\leq\bigg(\sum_{T\in \mathcal{T}_h}h_T^2 \| \b{f}\|^2_{L^2(T)}\bigg)^{1/2} 
    |\b{\phi}|_{H^1 (\Omega)}.
    \end{align*}
The bound on second and third terms follows from Cauchy-Schwartz, discrete trace inequality and Lemma \ref{lem:Interpolation} as:
    \begin{align*} 
    \sum_{e \in \mathcal{E}_h^F}\int_{e}(\b{g}-\b{\sigma}_h(\b{u_h})\b{n_e})\cdot (\b{\phi-\phi_h})~ds
    &\leq\bigg( \sum_{e \in \mathcal{E}_h^F}h_e\|\b{g}-\b{\sigma}_h(\b{u_h)n_e}\|^2_{L^2(e)}\bigg)^{1/2}\bigg(\sum_{e \in \mathcal{E}_h^F}h_e^{-1}\| \b{\phi-\phi_h}\|^2_{L^2(e)}\bigg)^{1/2}\\
    &\leq \bigg( \sum_{e \in \mathcal{E}_h^F}h_e\|\b{g}-\b{\sigma}_h(\b{u_h)n_e}
    \|^2_{L^2(e)}\bigg)^{1/2}|\b{\phi}|_{H^1 (\Omega)}.
   \end{align*}
   and

   \begin{align*}
   -\sum_{e\in \mathcal{E}_h^i}\int_{e}\sjump{\b{\sigma}_h(\b{u_h})}\cdot\smean{\b{\phi-\phi_h}}~ds
    &\leq\bigg(\sum_{e \in \mathcal{E}_h^i}h_e\|\sjump{\b{\sigma}_h(\b{u_h})}\|^2_{L^2(e)}\bigg)^{1/2}\bigg(\sum_{e \in \mathcal{E}_h^i}h_e^{-1}\|\b{\phi-\phi_h}\|^2_{L^2(e)}\bigg)^{1/2}\\
    &\leq\bigg(\sum_{e \in \mathcal{E}_h^i}h_e\|\sjump{\b{\sigma}_h(\b{u_h})}\|^2_{L^2(e)}\bigg)^{1/2}|\b{\phi}|_{H^1 (\Omega)}.
    \end{align*}
 As $\b{\phi_h} \in \b{V_c}$ the bound on $b_h(\b{u_h},\b{\phi_h})$ directly follows from \eqref{4.2}. Again a use of Cauchy-Schwartz, discrete-trace inequality and Lemma \ref{lem:Interpolation} yields
 \begin{align*}
  -\sum_{e\in \mathcal{E}_h^C}\int_{e}({\sigma}_{hn}(\b{u_{h}})&+ c_{n}(u_{hn}-g_a)_+^{m_n}) \cdot (\phi-\phi_h)_n~ds\\
  & \leq\bigg(\sum_{e \in \mathcal{E}_h^C}h_e\|{\sigma}_{hn}(\b{u_{h}})+ c_{n}(u_{hn}-g_a)_+^{m_n}\|^2_{L^2(e)}\bigg)^{1/2}\bigg(\sum_{e \in \mathcal{E}_h^C}h_e^{-1}\|\b{\phi-\phi_h}\|^2_{L^2(e)}\bigg)^{1/2}\\
    &\leq\bigg(\sum_{e \in \mathcal{E}_h^C}h_e\|{\sigma}_{hn}(\b{u_{h}})+ c_{n}(u_{hn}-g_a)_+^{m_n}\|^2_{L^2(e)}\bigg)^{1/2}|\b{\phi}|_{H^1 (\Omega)}.
 \end{align*}
 and
 \begin{align*}
 -\sum_{e \in \mathcal{E}_h^C} \int _e (\b{\sigma}_{h\tau}(\b{u_h})&+c_\tau \b{\lambda_{h\tau})\cdot (\phi-\phi_h)_\tau}~ds\\ &\leq 
 \bigg(\sum_{e \in \mathcal{E}_h^C}h_e\| \b{\sigma}_{h\tau}(\b{u_h})+c_\tau \b{\lambda_{h\tau}}\|^2_{L^2(e)}\bigg)^{1/2}\bigg(\sum_{e \in \mathcal{E}_h^C}h_e^{-1}\|(\b{\phi-\phi_h})_\tau\|^2_{L^2(e)}\bigg)^{1/2} \\
 &\leq \bigg(\sum_{e \in \mathcal{E}_h^C}h_e\| \b{\sigma}_{h\tau}(\b{u_h})+c_\tau \b{\lambda_{h\tau}}\|^2_{L^2(e)}\bigg)^{1/2}|\b{\phi}|_{H^1 (\Omega)}.
 \end{align*}
 Combining, we have
 \begin{align}\label{5.3}
     T_1 \leq \eta_h |\b{\phi}|_{H^1 (\Omega)}.
 \end{align}
 Using the boundedness of the bilinear form $B_h(\cdot, \cdot)$ w.r.t. $\norm{\cdot}_h$\and Lemma \ref{lem:EO}, we have
 \begin{align*}
  T_2= a_h(\b{u_h}-E_h\b{u_h},\b{\phi})
   \leq \norm{\b{u_h}-E_h\b{u_h}}_h \norm{\b{\phi}}_h \leq \eta_3|\b{\phi}|_{H^1{(\Omega)}}.  
 \end{align*}
Further, using the relation $|\b{\lambda_\tau}| \leq 1$, $|\b{\lambda_{h\tau}}| \leq 1$, $\b{\lambda_\tau} \cdot \b{u_\tau} = |\b{u_\tau}|$ 
and $\b{\lambda_{h\tau}} \cdot \b{u_{h\tau}} = |\b{u_{h\tau}}|$ a.e. on $\Gamma_C$, the term $T_3$ can be estimated as:
    \begin{align*}
     T_3 &= g(\b{\lambda_{h\tau},\phi})-g(\b{\lambda_\tau,\phi})\\
     &= \int_{\Gamma_C}c_{\tau}\b{\lambda_{h\tau}}\cdot \b{u_\tau}~ds -\int_{\Gamma_C}c_{\tau}\b{\lambda_{h\tau}}\cdot (E_h\b{u_{h}})_\tau~ds-
     \int_{\Gamma_C}c_{\tau}\b{\lambda_{\tau}}\cdot \b{u_\tau}~ds+
     \int_{\Gamma_C}c_{\tau}\b{\lambda_{\tau}}\cdot (E_h\b{u_{h}})_\tau~ds\\
     &\leq \int_{\Gamma_C} c_\tau |(E_h\b{u_h})_\tau|~ds-\int_{\Gamma_C}c_{\tau}\b{\lambda_{h\tau}}\cdot (E_h\b{u_{h}})_\tau~ds \\
     &= \int_{\Gamma_C} c_\tau |(E_h\b{u_h})_\tau|~ds-\int_{\Gamma_C} c_\tau |{\b{u_h}_\tau}|~ds+\int_{\Gamma_C}c_{\tau}\b{\lambda_{h\tau}}\cdot \b{u_{h\tau}}~ds -\int_{\Gamma_C}c_{\tau}\b{\lambda_{h\tau}}\cdot (E_h\b{u_{h}})_\tau~ds \\
     &\leq \int_{\Gamma_C} c_\tau |(E_h\b{u_h})_\tau-{\b{u_h}_\tau}|~ds + \int_{\Gamma_C}c_{\tau}\b{\lambda_{h\tau}}\cdot (\b{u_{h\tau}}- (E_h\b{u_{h}})_\tau)~ds\\
     &\leq 2\sum_{e\in \mathcal{E}_h^C}\|c_\tau\|_{L^2(e)}\|E_h\b{u_h}-\b{u_h}\|_{L^2(e)}\\
     &\leq C \bigg(\sum_{e\in \mathcal{E}_h^C} h_e \|c_\tau\|^2_{L^{2}(e)}\bigg)^{1/2}\eta_3.
     \end{align*}
In order to estimate $T_4$, we will use standard monotonicity argument \cite{Han:1993:FCPNC}~i.e. $((x-c)^r_+ - (y-c)^r_+)(x-y)\geq 0 ~\forall~ x,y,c \in \mathbb{R},~r\geq 0)$ to observe
\begin{align}\label{5.4}
j_n(\b{u_h},\b{u-u_h})-j_n(\b{u},\b{u-u_h}) \leq 0.
\end{align}
Thus, a use of \eqref{5.4} yields
    \begin{align*}
     T_4&= j_n(\b{u_h},\b{\phi})-j_n(\b{u,\phi})\\
    &= j_n(\b{u_h}, \b{u}-\b{u_h}) + j_n(\b{u_h}, \b{u_h}-E_h\b{u_h})- j_n(\b{u}, \b{u}-\b{u_h}) -j_n(\b{u}, \b{u_h}-E_h\b{u_h})\\ &\leq j_n(\b{u_h}, \b{u_h}-E_h\b{u_h}) -j_n(\b{u},\b{u_h}-E_h\b{u_h})\\
    &=\sum_{e\in \mathcal{E}^C_h}\int_e c_n[(u_{hn}-g_a)^{m_n}_+-({u}_n-g_a)^{m_n}_+](u_{h}-E_h{u_h})_n~ds.
    \end{align*}
We will consider two different cases: $m_n=~1$ and $m_n>1$. When $m_n = 1$, the last relation reduces to
\begin{align*}
T_4 &\leq \sum_{e\in \mathcal{E}^C_h}\int_e c_n[(u_{hn}-g_a)_+-({u}_n-g_a)_+](u_{h}-E_h{u_h})_n~ds\\
&\leq \sum_{e\in \mathcal{E}^C_h}\int_e c_n|u_{hn} -{u}_n|(u_{h}-E_h{u_h})_n~ds \\
&\leq \sum_{e\in \mathcal{E}^C_h} \|c_n\|_{L^{\infty}(e)}\|u_{hn} -{u}_n\|_{L^2(e)}\|(u_{h}-E_h{u_h})_n\|_{L^2(e)}\\
&\leq \sum_{e\in \mathcal{E}^C_h} \|c_n\|_{L^{\infty}(e)}\|\b{u_{h} -{u}}\|_{1,h}\|(u_{h}-E_h{u_h})_n\|_{L^2(e)}.
\end{align*}
otherwise for $m_n > 1$, using the identity $|(a)^m_+-(b)^m_+|\leq m|a-b|\big(|a|^{m-1}+|b|^{m-1}\big)~ a,~b \in \mathbb{R}, m\geq1$, Cauchy H\"older's inequality, \eqref{im} together with Lemma \ref{lem:ubound}  and Lemma \ref{lem:uhbound} , we find
   \begin{align*}
    T_4 &\leq \sum_{e\in \mathcal{E}^C_h} \|c_n\|_{L^{\infty}(e)}\|(u_{hn}-g_a)^{m_n}_+-({u}_n-g_a)^{m_n}_+\|_{L^2(e)}\|u_{hn}-(E_h\b{u_h})_n\|_{L^2(e)}\\
    &\leq \sum_{e\in \mathcal{E}^C_h} C \bigg(\|u_{hn}-g_a\|^{m_n-1}_{L^q(m_n-1)(e)}+\|u_{n}-g_a\|^{m_n-1}_{L^q(m_n-1)(e)}\bigg)\|u_n-u_{hn}\|_{L^p(e)}\\ &~~~~~~~~~~\|u_{hn}-(E_h\b{u_h})_n\|_{L^2(e)}~ds\\
    &\leq  \sum_{e\in \mathcal{E}^C_h} C \bigg(\|\b{u_{h}}-g_a\|^{m_n-1}_{1,h}+\|\b{u}-g_a\|^{m_n-1}_{H^1(\Omega)}\bigg)\|u_n-u_{hn}\|_{L^p(e)}\|u_{hn}-(E_h\b{u_h})_n\|_{L^2(e)}~ds \\
     &\leq \sum_{e\in \mathcal{E}^C_h} C\|\b{u_{h} -{u}}\|_{1,h}\|(u_{h}-E_h{u_h})_n\|_{L^2(e)}.
    \end{align*}
where the H\"older conjugates $\frac{p}{2}, \frac{q}{2} \in (1,\infty)$ satisfying $\frac{1}{p}+\frac{1}{q}=\frac{1}{2}$ are such that $q(m_n-1)\geq 1$.
Thus, for $m_n \geq 1$, we have
\begin{align*}
 T_4 \leq  C\|\b{u-u_{h}}\|_{1,h}\sum_{e\in \mathcal{E}^C_h}\|u_{hn}-(E_h\b{u_h})_n\|_{L^2(e)}.  
\end{align*}
Finally, using standard inverse estimate and discrete Cauchy- Schwartz inequality in $T_4$, we obtain
\begin{align*}
T_4 &\leq  C\|\b{u-u_{h}}\|_{1,h}\sum_{e\in \mathcal{E}^C_h}\|u_{hn}-(E_h\b{u_h})_n\|_{L^2(e)}\\
&\leq  C\|\b{u-u_{h}}\|_{1,h}\sum_{e\in \mathcal{E}^C_h}\sum_{T\in \mathcal{T}_e}{h_e}^{-1/2}\|\b{u_{h}}-E_h\b{u_h}\|_{L^2(T)}\\
&\leq C\|\b{u-u_{h}}\|_{1,h} \bigg(\sum_{e\in \mathcal{E}^C_h} h_e\bigg)^{1/2} \bigg(\sum_{T\in \mathcal{T}_h}h_T^{-2}\|\b{u_{h}}-E_h\b{u_h}\|^2_{L^2(T)} \bigg)^{1/2}. 
\end{align*}
as $h_e \sim h_T$. As a consequence of Lemma 2.4 and identity $\sum_{e\in \mathcal{E}_h^C} |h_e|=|\Gamma_C|$, we find
 \begin{align*}
 T_4 &\leq C\norm{\b{u-u_h}}_h\eta_3.
 \end{align*}
 Combining the estimates obtained in $T_1,~T_2, ~T_3,~T_4$ and using Young's inequality, we get the desired bound on the error term.
\par
In order to find the upper bound for $\sum_{e \in \mathcal{E}_h^C}h_e\|\sigma_n(\b{u-u_h})\|^2_{L^2(e)}$, we recall $\eqref{1.7}$, and use identity $(a+b)^2 \leq 2(a^2+b^2)$ as follows:
\begin{align}\label{5.6}
h_e\|\sigma_n(\b{u-u_h})\|^2_{L^2(e)} &\leq 2(h_e\|-c_n(u_n-g_a)^{m_n}_+ + c_n(u_{hn}-g_a)^{m_n}_+\|^2_{L^2(e)} \nonumber\\&+ h_e\|\sigma_n(\b{u_h}) + c_n(u_{hn}-g_a)^{m_n}_+\|^2_{L^2(e)})
\end{align}
where $e \in \mathcal{E}_h^C$. Using the similar arguments as used in estimating $T_4$, we get
\begin{align*}
\|-c_n(u_n-g_a)^{m_n}_+ + c_n(u_{hn}-g_a)^{m_n}_+\|_{L^2(e)}  \leq C \|u_n - u_{hn}\|_{L^p(e)}.
\end{align*}

for $p>2$. As a consequence, we find
\begin{align}\label{dm}
\|-c_n(u_n-g_a)^{m_n}_+ + c_n(u_{hn}-g_a)^{m_n}_+\|_{L^2(e)}  \leq C \|u_n - u_{hn}\|_{L^p(\Gamma_C)}\leq C\|\b{u-u_h}\|_{1,h}.  
\end{align}
Therefore, summing $\eqref{5.6}$ over all $e \in \mathcal{E}^C_h$ and using the identity $\sum_{e\in \mathcal{E}_h^C} |h_e|=|\Gamma_C|$, we find
\begin{align}\label{5.7}
 \sum_{e \in \mathcal{E}^C_h}h_e\|\sigma_n(\b{u-u_h})\|^2_{L^2(e)} \leq \norm{\b{u-u_h}}^2_h + \eta_6^2. 
\end{align}
This completes the proof.
\end{proof}
\subsection{Efficiency estimates}
In this section, we show that the error estimator $\eta_h$ provides a lower bound for the true error up to data oscillations. In order to prove the efficiency of the estimators we will first prove the following lemma.
\begin{lemma} Let $\b{u}\in \b{V}$ be the solution of continuous problem (1.10) and let $\b{v_h}\in \b{V_h}$ be an arbitrary element then, the following results hold:
  \begin{align*}
 &(i) \sum_{T \in \mathcal{T}_h}h_T^2 \|\b{f}\|^2_{L^2(T)} \leq C(\norm{\b{u-v_h}}^2_h + Osc(\b{f})^2), \\
  &(ii) \sum_{e \in \mathcal{E}^i_h}h_e \|\sjump{\b{\sigma}_h(\b{v_h})}\|^2_{L^2(e)} \leq C(\norm{\b{u-v_h}}^2_h + Osc(\b{f})^2), \\
  &(iii) \sum_{e \in \mathcal{E}^F_h}h_e \|\b{\sigma}_h(\b{v_h})\b{n}-\b{g}\|^2_{L^2(e)} \leq C(\norm{\b{u-v_h}}^2_h + Osc(\b{f})^2+Osc(\b{g})^2),\\
 &(iv) \sum_{e \in \mathcal{E}^C_h}h_e \|\b{\sigma}_{h\tau}(\b{v_h})+c_\tau \b{\lambda_{\tau}}\|^2_{L^2(e)} \leq C(\norm{\b{u-v_h}}^2_h + Osc(\b{f})^2+Osc(c_\tau)^2+Osc(\b{\lambda_\tau})^2),\\
   &(v)\sum_{e \in \mathcal{E}^C_h} h_e\|\b{\sigma}_{hn}(\b{v_h})+c_n(v_{hn}-g_a)^{m_n}_+\|^2_{L^2(e)}
   \leq C (\norm{\b{v_h-u}}_h^2 +\|\b{v_h-u}\|^{2m_n}_{1,h} + \sum_{e \in \mathcal{E}^C_h}h_e\|\b{\sigma}_{hn}(\b{v_h-u})\|^2_{L^2(e)}),  
   \end{align*}\\
where
    \begin{align*}
    Osc(\b{f})^2&=\sum_{T \in \mathcal{T}_h} h^2_T\|\b{f-\bar{f}}\|^2_{L^2(T)},\\
    Osc(\b{g})^2&=\sum_{e \in \mathcal{E}^F_h}h_e\|\b{g-\bar{g}}
    \|^2_{L^2(e)},\\
    Osc(c_\tau)^2&=\sum_{e \in \mathcal{E}^C_h}h_e\|c_\tau-\bar{c_\tau}\|^2_{L^2(e)},\\
    Osc(\b{\lambda_\tau})^2&=\sum_{e \in \mathcal{E}^C_h}h_e\|\b{\lambda_\tau-\bar{\lambda_\tau}}\|^2_{L^2(e)}.
    \end{align*}
where $\b{\bar{v}}$ denotes the $L^2$ projection of $\b{v}$ onto the space of piece-wise constant functions.
\end{lemma}
\begin{proof}
\textbf{($i$)} Let $T \in \mathcal{T}_h $ be arbitrary and let $\xi \in P_3(T)$ be bubble function that vanishes on $\partial T$ and takes unit value at the barycenter of $T$. By equivalence of norms on finite dimensional spaces, we have 
    \begin{align}\label{5.8}
     \|\b{\bar{f}}\|^2_{L^2(T)} \leq \int_{T} \xi\b{\bar{f}}\cdot \b{\bar{f}}~dx.
    \end{align}
Let $\b{\phi = \bar{f}} \xi$. We can identify $\b{\phi}$ as an element of $[H^1_0(\Omega)]^2$ by extending it by 0 outside of $T$. It follows from Lemma 1.1, integration by parts and a standard inverse estimate that
    \begin{align}\label{5.9}
    \int_{T} \xi\bar{\b{f}}\cdot \b{\bar{f}}~dx &= \int_{\Omega} \b{{f}}\cdot \b{\phi}~dx + \int_{\Omega} 
    \b{(\bar{f}-f)}\cdot \b{\phi}~dx \nonumber \\
    &= a(\b{u,\phi}) + \int_{\Omega} (\b{\bar{f}-f})\cdot \b{\phi}~dx + \int_{T}\b{div} \b{\sigma}(\b{v_h})\cdot\b{\phi}~dx \nonumber \\
    &= a(\b{u,\phi}) + \int_{T} (\b{\bar{f}-f)}\cdot \b{\phi}~dx - \int_{T} \b{\sigma}(\b{v_h}): \b{\epsilon}(\b{\phi})~dx \nonumber \\
    &= \int_{T}(\b{\sigma}(\b{u})-\b{\sigma}(\b{v_h})): \b{\epsilon}(\b{\phi})~dx +\int_{T} (\b{\bar{f}-f)}\cdot \b{\phi}~dx \nonumber \\
    &\leq |\b{u-v_h}|_{H^1(T)} |\b{\phi}|_{H^1(T)} + \|\b{\bar{f}-f}\|_{L^2(T)}\| \b{\phi}\|_{L^2(T)} \nonumber \\
    &\leq |\b{u-v_h}|_{H^1(T)} h_T^{-1}\|\b{\phi}\|_{L^2(T)} + \|\b{\bar{f}-f}\|_{L^2(T)}\| \b{\phi}\|_{L^2(T)} \nonumber \\
    &\leq \big(|\b{u-v_h}|_{H^1(T)} h_T^{-1} + \|\b{\bar{f}-f}\|_{L^2(T)} \big)\|\b{\bar{f}}\|_{L^2(T)}.
    \end{align}
Combining \eqref{5.8} and \eqref{5.9}, we obtain
   \begin{align*}
   h^2_T\|\b{\bar{f}}\|^2_{L^2(T)} \leq |\b{u-v_h}|^2_{H^1(T)} + h^2_T\|\b{\bar{f}-f}\|^2_{L^2(T)}.
   \end{align*}
and hence by triangle inequality,
\begin{align}\label{5.10}
h^2_T\|\b{f}\|^2_{L^2(T)} \leq |\b{u-v_h}|^2_{H^1(T)} +  h^2_T\|\b{\bar{f}-f}\|^2_{L^2(T)}.    
\end{align}
Summing up \eqref{5.10} over all triangles in $\mathcal{T}_h$ we get the desired result.
  
\smallbreak
\textbf{($ii$)} 
Let $e \in \mathcal{E}^i_h$ be arbitrary and this edge is shared by two triangles $T^{-}$ and $T^{+}$. Let $n_e$ be the unit vector normal to $e$ and pointing from the triangle $T^-$ to $T^+$. We construct a bubble function $\xi \in P_4(T^- \cup T^+)$  such that it vanishes on the boundary of quadrilateral $T^- \cup T^+$ and takes unit value at the midpoint  of $e$. Define $\b{\beta}= \xi \b{\xi_1} $ on $T^- \cup T^+$ where $\b{\xi_1} \in [P_0(T^{-} \cup T^{+})]^2$ such that $\b{\xi_1}=\sjump{\b{\sigma}_h(\b{v_h})}$ on edge e. We can identify $\b{\beta}$ by its zero extension outside $T^{-} \cup T^{+}$ yielding $\b{\beta} \in [H^1_0(\Omega)]^2$. A use of equivalence of norms on finite dimensional space yields 
    \begin{align}\label{5.11}
    \|\b{\xi_1}\|^2_{L^2(e)} &\leq \int_e \xi\b{\xi_1}\cdot \b{\xi_1}~ds \nonumber\\
    &=\int_e \b{\beta} \cdot \b{\xi_1}~ds.
    \end{align}
It then follows from integration by parts, Lemma 1.1, Cauchy Schwartz inequality and standard inverse estimate that
    \begin{align}\label{5.12}
    \int_{e}\sjump{\b{\sigma}_h(\b{v_h})}\cdot \b{\beta}~ds &= \int_{T^-\cup T^+} \b{\sigma}_h(\b{v_h}):\b{\epsilon}(\b{\beta})~dx \nonumber\\
    &= \int_{T^-\cup T^+} \b{\sigma}_h(\b{v_h}):\b{\epsilon}(\b{\beta})~dx  + \int_{T^-\cup T^+} \b{f}\cdot \b{\beta}~dx -\int_{T^-\cup T^+} \b{\sigma}(\b{u}):\b{\epsilon}(\b{\beta})~dx \nonumber \\
    &\leq \sum_{T \in \mathcal{T}_e}\bigg(|\b{u-v_h}|_{H^1(T)}|\b{\beta}|_{H^1(T)} + \|\b{f}\|_{L^2(T)}\|\b{\beta}\|_{L^2(T)}\bigg) \nonumber \\
    &\leq \sum_{T \in \mathcal{T}_e}\big(|\b{u-v_h}|_{H^1(T)}h^{-1}_T + \|\b{f}\|_{L^2(T)}\big)\|\b{\beta}\|_{L^2(T)}  \nonumber\\
     &\leq \sum_{T \in \mathcal{T}_e}\big(|\b{u-v_h}|_{H^1(T)}h^{-1}_T + \|\b{f}\|_{L^2(T)}\big) h^{1/2}_e \|\b{\xi_1}\|_{L^2(e)}.
    \end{align}
Since $h_e \sim h_T$, therefore, combining \eqref{5.11} and \eqref{5.12}, we obtain
    \begin{align}\label{5.13}
    h_e^{1/2}\|\sjump{\b{\sigma}_h(\b{v_h})}\|_{L^2(e)} &\leq \sum_{T \in \mathcal{T}_e}\big(|\b{u-v_h}|_{H^1(T)} +h_T \|\b{f}\|_{L^2(T)}\big).
    \end{align}
Squaring \eqref{5.13} and summing up over all the interior edges, we find
    \begin{align*}
    \sum_{e \in \mathcal{E}_h^i}h_e\|\sjump{\b{\sigma}_h(\b{v_h})}\|^2_{L^2(e)}&\leq \sum_{T \in \mathcal{T}_h}|\b{u-v_h}|^2_{H^1(T)} +\sum_{T \in \mathcal{T}_h}h^2_T \|\b{f}\|^2_{L_2(T)}.
     \end{align*}
Finally, ($ii$) follows with a use of ($i$).
\\
\\
\textbf{($iii$)} Let $e \in \mathcal{E}^F_h$ and let $T$ be the triangle such that $e \subseteq \partial{T}$. We construct a bubble function $\xi \in P_2(T)$ that vanishes on $\partial {T} \setminus e$ and takes unit value at the midpoint  of $e$. Define $\b{\xi_1} \in [P_0(T)]^2$ by assigning $\b{\xi_1} =\b{\sigma}_h(\b{v_h})\b{n}-\b{\bar{g}}$ on edge e. Define $\b{\beta}= \xi \b{\xi_1} $ on $T$ and extend $\b{\beta}$ by 0 outside of T and hence it belongs to $\b{V}$. Now, using equivalence of norms on finite dimensional space, we obtain
    \begin{align}\label{5.14}
    \|\b{\xi_1}\|^2_{L^2(e)} &\leq \int_e \xi\b{\xi_1} \cdot \b{\xi_1}~ds\nonumber\\
    &=\int_e \b{\beta} \cdot \b{\xi_1}~ds.  
    \end{align}
Using Lemma \ref{ref:Cchar}, we find
   \begin{align}\label{5.15}
   \int_e \b{\beta} \cdot \b{\xi_1}~ds &=\int_e \b{\sigma}_h(\b{v_h})\b{n}\cdot \b{\beta}~ds+\int_e (\b{g-\bar{g}})\cdot \b{\beta}~ds -\int_e \b{g}\cdot \b{\beta}~ds \nonumber\\
   &=\int_e \b{\sigma}_h(\b{v_h})\b{n}\cdot \b{\beta}~ds +\int_e (\b{g-\bar{g}})\cdot \b{\beta}~ds -a(\b{u,\beta)}+\int_T \b{f\cdot \beta}~dx.
   \end{align}
Now, the use of integration by parts, Cauchy Schwartz and standard inverse estimates in \eqref{5.15} yields
   \begin{align}\label{5.16}
   \int_e \b{\beta} \cdot \b{\xi_1}~ds &=  \int_{T} \b{\sigma}(\b{v_h}):\b{\epsilon}(\b{\beta})~dx-\int_{T} \b{\sigma}(\b{u}):\b{\epsilon}(\b{\beta})~dx+\int_e \b{(g-\bar{g})}\cdot \b{\beta}~ds+\int_T \b{f} \cdot \b{\beta}~dx.\nonumber\\
    & \leq |\b{u-v_h}|_{H^1(T)}  |\b{\beta}|_{H^1(T)}+\|\b{f}\|_{L^2(T)}\|\b{\beta}\|_{L^2(T)}+\|\b{g-\bar{g}}\|_{L^2(e)}\|\b{\beta}\|_{L^2(e)}\nonumber \\
    & \leq \big (|\b{u-v_h}|_{H^1(T)}h^{-1}_T + \|\b{f}\|_{L^2(T)}+h^{-1/2}_e\|\b{g-\bar{g}}\|_{L^2(e)}\big)\|\b{\beta}\|_{L^2(T)} \nonumber\\
    & \leq \bigg(|\b{u-v_h}|_{H^1(T)}h^{-1}_T+\|\b{f}\|_{L^2(T)}+h^{-1/2}_e\|\b{g-\bar{g}}\|_{L^2(e)}\bigg)h^{1/2}_e\|\b{\xi_1}\|_{L^2(e)}.
    \end{align}
Combining \eqref{5.14} and \eqref{5.16}, we get,
    \begin{align}\label{5.17}
    h^{1/2}_e\|\b{\xi_1}\|_{L^2(e)} &\leq |\b{u-v_h}|_{H^1(T)}+h_T\|\b{f}\|_{L^2(T)}+h^{1/2}_e\|\b{g-\bar{g}}\|_{L^2(e)}. 
    \end{align}
Squaring \eqref{5.17} and summing up over all $e \in \mathcal{E}^F_h$, we obtain
\begin{align*}
     \sum_{e \in \mathcal{E}_h^F}h_e\|\b{\sigma}_h(\b{v_h})\b{n}-\bar{\b{g}}\|^2_{L^2(e)}&\leq \sum_{T \in \mathcal{T}_h}|\b{u-v_h}|^2_{H^1(T)}+\sum_{T \in \mathcal{T}_h}h_T^2\|\b{f}\|^2_{L^2(T)}+\sum_{e \in \mathcal{E}_h^F}h_e\|\b{g-\bar{g}}\|^2_{L^2(e)}.
\end{align*}
hence, thereafter using triangle inequality $(iii)$ follows from $(i)$.
\smallbreak
\textbf{($iv$)} Let $e \in \mathcal{E}^C_h$ be arbitrary and let $T$ be the triangle such that $e \subseteq \partial{T}$.
 In order to estimate $\|\b{\sigma}_{h\tau}(\b{v_h})+c_\tau \b{\lambda_{\tau}}\|_{L^2(e)}$, we will make use of triangle inequality as follows:
    \begin{align}\label{5.18}
    \|\b{\sigma}_{h\tau}(\b{v_h})+c_\tau \b{\lambda_{\tau}}\|_{L^2(e)}\leq\|\b{\sigma}_{h\tau}(\b{v_h})+\bar{c_\tau}\b{ \lambda_{\tau}}\|_{L^2(e)}+\|\b{\lambda_{\tau}}(c_\tau-\bar{c_\tau})\|_{L^2(e)}.
    \end{align}
Also,
    \begin{align}\label{5.19}
    \|\b{\sigma}_{h\tau}(\b{v_h})+\bar{c_\tau} \b{\lambda_{\tau}}\|_{L^2(e)}\leq \|\b{\sigma}_{h\tau}(\b{v_h})+\bar{c_\tau} \b{\bar{\lambda_{\tau}}}\|_{L^2(e)}+\|\bar{c_\tau}\b{(\lambda_\tau-\bar{\lambda_\tau}})\|_{L^2(e)}.
    \end{align}
     Define a bubble function $\xi \in P_2(T)$  which vanishes on $\partial {T} \setminus e$ and takes unit value at the midpoint  of e. Let $\b{\xi_1} \in [P_0(T)]^2$ such that  $\xi_{1n}=0$  and  $\b{\xi_{1\tau}}=\b{\sigma}_{h\tau}(\b{v_h})+\bar{c_\tau} \b{\bar{\lambda_{\tau}}}$ on edge $e$. Define $\b{\beta} = \xi\b{\xi_1}$ on $T$ whose extension by 0 outside of $T$  belongs to $\b{V}$. Using the equivalence of norms on finite dimensional space, we have
    \begin{align}\label{5.20}
    \|\b{\xi_1}\|^2_{L^2(e)}&\leq \int_e \xi\b{\xi_1} \cdot \b{\xi_1}~ds \nonumber\\
     &= \int_e  \b{\sigma}_{h\tau}(\b{v_h})\cdot \b{\beta}~ds+\int_e \bar{c_\tau} \b{\bar{\lambda_{\tau}}} \cdot \b{\beta}~ds.
     \end{align}
as $\b{\xi_{1\tau}}=\b{\xi_1}$. A use of integration by parts yields
     \begin{align}\label{5.21}
     \int_{T} \b{\sigma}_h(\b{v_h}):\b{\epsilon}(\b{\beta})~dx&=\int_{T} -\b{div} \b{\sigma}_h(\b{v_h})\cdot \b{\beta}~dx +\int_{\partial{T}}\b{\sigma}_h(\b{v_h})\b{n}\cdot\b{\beta}~ds \nonumber\\
     &=\int_{e}\b{\sigma}_{hn}(\b{v_h})\beta_n~ds + \int_{e}\b{\sigma}_{h\tau}(\b{v_h})\cdot \b{\beta_{\tau}}~ds \nonumber\\
     &=\int_{e}\b{\sigma}_{h\tau}(\b{v_h})\cdot \b{\beta_{\tau}}~ds 
     \nonumber\\
     &=\int_{e}\b{\sigma}_{h\tau}(\b{v_h})\cdot \b{\beta}~ds.
     \end{align}
as $\beta_n=0$. Now, it follows from \eqref{5.20}, \eqref{5.21}, Lemma 1.1, Cauchy Schwartz inequality and standard inverse estimate that 
     \begin{align*}
    \|\b{\xi_1}\|^2_{L^2(e)}&\leq \int_{T} \b{\sigma}_h(\b{v_h}):\b{\epsilon}(\b{\beta})~dx+\int_e \bar{c_\tau} \b{\bar{\lambda_{\tau}}} \cdot \b{\beta}~ds \nonumber \\
     &= \int_{T} \b{\sigma}_h(\b{v_h}):\b{\epsilon}(\b{\beta})~dx -\int_{T} \b{\sigma}(\b{u}):\b{\epsilon}(\b{\beta})~dx -\int_e c_\tau \b{\lambda_{\tau}} \cdot \b{\beta}~ds + \int_T \b{f}\cdot \b{\beta}~dx 
     +\int_e \bar {c_\tau}\b{\bar{\lambda_{\tau}}} \cdot \b{\beta}~ds \nonumber\\
     &= \int_{T} (\b{\sigma}_h(\b{v_h}) - \b{\sigma}(\b{u})):\b{\epsilon}(\b{\beta})~dx+\int_T \b{f}\cdot \b{\beta}~dx + \int_e c_\tau \b{(\bar{\lambda_{\tau}}}-\b{\lambda_{\tau}) \cdot \beta}~ds + \int_e \b{\bar{\lambda_{\tau}}} (\bar{c_{\tau}}-c_{\tau}) \cdot \b{\beta}~ds\nonumber \\
     &\leq |\b{u-v_h}|_{H^1(T)}|\b{\beta}|_{H^1(T)}+ \|\b{f}\|_{L^2(T)}\|\b{\beta}\|_{L^2(T)}+\|c_\tau\|_{L^{\infty}(e)}\|\b{\bar{\lambda_{\tau}}-\lambda_{\tau}}\|_{L^{2}(e)}\|\b{\beta}\|_{L^2(e)}\nonumber\\
     &+ \|\b{\bar{\lambda_{\tau}}}\|_{L^{\infty}(e)}\|\bar{c_{\tau}}-c_{\tau}\|_{L^2(e)}\|\b{\beta}\|_{L^2(e)}\nonumber\\
     &\leq \bigg(|\b{u-v_h}|_{H^1(T)}h^{-1}_T + \|\b{f}\|_{L^2(T)}+h^{-1/2}_e\|c_\tau\|_{L^{\infty}(e)}\|\b{\bar{\lambda_{\tau}}-\lambda_{\tau}}\|_{L^2(e)}\nonumber\\
     &+ h^{-1/2}_e\|\b{\bar{\lambda_{\tau}}}\|_{L^{\infty}(e)}\|\bar{c_{\tau}}-c_{\tau}\|_{L^2(e)}\bigg)\|\b{\beta}\|_{L^2(T)}\nonumber\\
     &\leq h^{1/2}_e\bigg(|\b{u-v_h}|_{H^1(T)}h^{-1}_T + \|\b{f}\|_{L^2(T)}+h^{-1/2}_e\|c_{\tau}\|_{L^{\infty}(e)}\|\b{\bar{\lambda_{\tau}}-\lambda_{\tau}}\|_{L^{2}(e)} \nonumber\\
     &+ h^{-1/2}_e\|\bar{c_{\tau}}-c_{\tau}\|_{L^2(e)}\bigg)\|\b{\xi_1}\|_{L^2(e)}.
    \end{align*}
Squaring the last equation and summing over all $e \in \mathcal{E}^C_h$, we obtain
   \begin{align}\label{5.22}
    \sum_{e \in \mathcal{E}^C_h}h_e \|\b{\sigma}_{h\tau}(\b{v_h})+\bar{c_\tau} \b{\bar{\lambda_{\tau}}}\|^2_{L_2(e)} &\leq   C(\norm{\b{u-v_h}}^2_h +  Osc(\b{f})^2+Osc(c_\tau)^2+Osc(\b{\lambda_\tau})^2  
   \end{align}
Finally using \eqref{5.18}, \eqref{5.19} and \eqref{5.22}, we arrive at the desired estimate ($iv$).
\\
\\
\textbf{($v$)} This term can not be estimated directly by using the standard techniques of bubble functions.

Since, in general due to the positive part of the function,
    \begin{align*}
    \|\b{\sigma}_{hn}(\b{v_h})+c_n(v_{hn}-g_a)^{m_n}_+\|^2_{L^2(e)}\nleq \int_{e}(\b{\sigma}_{hn}(\b{v_h})+c_n(v_{hn}-g_a)^{m_n}_+)^2\b{\beta}~ds
    \end{align*}
where $\b{\beta}$ is an edge bubble function. We proceed to estimate it as follows: first using \eqref{1.7}, we find
$\|{\sigma}_{hn}(\b{v_h})+c_n(v_{hn}-g_a)^{m_n}_+\|_{L^2(e)}$
    \begin{align}\label{5.23}
    &=\|{\sigma}_{hn}(\b{v_h})-{\sigma}_n(\b{u})-c_n(u_n-g_a)^{m_n}_+ + c_n(v_{hn}-g_a)^{m_n}_+\|_{L^2(e)} \nonumber\\
    &\leq \|{\sigma}_{hn}(\b{v_h})-{\sigma}_n(\b{u})\|_{L^2(e)}+\| c_n(v_{hn}-g_a)^{m_n}_+ -c_n(u_n-g_a)^{m_n}_+ \|_{L^2(e)}.
    \end{align}
where $e \in \mathcal{E}^C_h$. Again, we will consider two cases. For $m_n = 1$, using  $\eqref{im}$ in the above equation $\eqref{5.23}$, we find
\begin{align*}
\|{\sigma}_{hn}(\b{v_h})+c_n(v_{hn}-g_a)^{m_n}_+\|_{L^2(e)} &\leq \|{\sigma}_{hn}(\b{v_h})-{\sigma}_n(\b{u})\|_{L^2(e)}+\|c_n\|_{L^{\infty}(e)}\|v_{hn} - u_n\|_{L^2(e)} \\
&\leq \|{\sigma}_{hn}(\b{v_h})-{\sigma}_n(\b{u})\|_{L^2(e)}+\|c_n\|_{L^{\infty}(e)}\|\b{v_{h} - u}\|_{1,h}. 
\end{align*}
Otherwise for $m_n>1$, a use of Cauchy H\"older's inequality and identity $|a^m-b^m|\leq m~|a-b|~(|a|^{m-1}+|b|^{m-1})~\text{where}~ a,b\geq~0,~m\geq1$ yields \\
 $\|{\sigma}_{hn}(\b{v_h})+c_n(v_{hn}-g_a)^{m_n}_+\|_{L^2(e)}$
    \begin{align}\label{pq}
    &\leq m_n\| c_n\|_{L^{\infty}(e)}\|(v_{hn}-u_n)(|v_{hn}-g_a|^{m_n-1}+|u_n-g_a|^{m_n-1}) \|_{L^2(e)}\nonumber\\&+ \|{\sigma}_{hn}(\b{v_h-u})\|_{L^2(e)}\nonumber\\
    &\leq m_n\| c_n\|_{L^{\infty}(e)}\|v_{hn}-u_n\|_{L^p(e)}\||v_{hn}-g_a|^{m_n-1}+|u_n-g_a|^{m_n-1} \|_{L^q(e)}\nonumber\\&+ \|{\sigma}_{hn}(\b{v_h-u)}\|_{L^2(e)}\nonumber\\
    &\leq C\|v_{hn}-u_n\|_{L^p(e)}\bigg(\|v_{hn}-g_a\|^{m_n-1}_{L^{q(m_n-1)}(e)}+ \|u_n-g_a\|^{m_n-1}_{L^{q(m_n-1)}(e)} \bigg)\nonumber\\&+ \|{\sigma}_{hn}(\b{v_h-u})\|_{L^2(e)}.
     \end{align}
where, $\frac{p}{2}$ and $\frac{q}{2}$ are H\"older conjugates satisfying $\frac{1}{p}+\frac{1}{q}=\frac{1}{2}$ and $q(m_n-1)\geq1$. Further, using \eqref{im} and Lemma \ref{lem:ubound}, we obtain
   \begin{align}\label{5.24}
   \|u_n-g_a\|^{m_n-1}_{L^{q(m_n-1)}(e)} \leq \|u_n-g_a\|^{m_n-1}_{L^{q(m_n-1)}(\Gamma_C)} \leq  C.
   \end{align}
Also, 
    \begin{align}\label{5.25}
    \|v_{hn}-g_a\|_{L^{q(m_n-1)}(e)} &\leq  \|v_{hn}-u_n\|_{L^{q(m_n-1)}(\Gamma_C)}\nonumber \\
    &\leq \|v_{hn}-u_n\|_{L^{q(m_n-1)}(\Gamma_C)}+\|u_n-g_a\|_{L^{q(m_n-1)}(\Gamma_C)} \nonumber\\
    &\leq \|\b{v_{h}-u}\|_{1,h}+C.
    \end{align}
where $C$ is constant depending on $\|\b{f}\|_{L^2(\Omega)}, \|\b{g}\|_{L^2(\Gamma_F)} \text{and}~\|g_a\|^{m_n-1}_{L^{q(m_n-1)}(\Gamma_C)} $. Using \eqref{5.24}, \eqref{im} and \eqref{5.25} in \eqref{pq} we obtain,
    \begin{align*}
    \|{\sigma}_{hn}(\b{v_h})+c_n(v_{hn}-g_a)^{m_n}_+\|_{L^2(e)}
    &\leq \|\b{v_{h}-u}\|_{1,h}\big(\|\b{v_{h}-u}\|^{m_n-1}_{1,h}+C \big) +\|{\sigma}_{hn}(\b{v_h-u})\|_{L^2(e)}.
    \end{align*}
Thus, we have
   \begin{align}\label{5.26}
   {h_e}^{1/2}\|{\sigma}_{hn}(\b{v_h})+c_n(v_{hn}-g_a)^{m_n}_+\|_{L^2(e)}
    &\leq h_e^{1/2}C\|\b{v_h-u}\|_{1,h} +h_e^{1/2}\|\b{v_h-u}\|^{m_n}_{1,h}\\
    & \quad+ h_e^{1/2}\|{\sigma}_{hn}(\b{v_h-u})\|_{L^2(e)}.\nonumber
   \end{align}
Squaring \eqref{5.26} and summing over all $e \in \mathcal{E}^C_h$ and finally using the identity $\sum_{e \in \mathcal{E}_h^C}|h_e| = \lvert \Gamma_C \rvert$, we obtain
   \begin{align*}
   \sum_{e \in \mathcal{E}^C_h} h_e\|{\sigma}_{hn}(\b{v_h})+c_n(v_{hn}-g_a)^{m_n}_+\|^2_{L^2(e)}
   &\leq C\norm{\b{v_h-u}}_h^2 +\|\b{v_h-u}\|^{2m_n}_{1,h} + \sum_{e \in \mathcal{E}^C_h}h_e\|\sigma_{hn}(\b{v_h-u})\|^2_{L^2(e)}.
   \end{align*}
   This completes the proof of this lemma.
 \end{proof}
 The following theorem ensures the efficiency of the error estimator $\eta_h$.
\begin{theorem} Let $\b{u} \in \b{V}$ and $\b{u_h} \in \b{V_h}$ be the solution of continuous problem \eqref{1.10} and discrete problem \eqref{4.1}, respectively. Then, the following results hold.
\begin{align*}
    \eta_h^2 &\leq \norm{\b{u_{h}-u}}^2_h +\sum_{e \in \mathcal{E}^C_h}h_e\|{\sigma}_{hn}(\b{u_h-u})\|^2_{L^2(e)} +\sum_{e \in \mathcal{E}^C_h}h_e\|c_\tau\|^2_{L^{\infty}(e)}\|\b{ \lambda_{h\tau}-\lambda_{\tau}}\|^2_{L^2(e)}\\&+ Osc(\b{f})^2 +Osc(c_\tau)^2+Osc(\b{\lambda_\tau})^2.
\end{align*}
\end{theorem}
\begin{proof}
 As $\eta_h^2=\eta_1^2+\eta_2^2+\eta_3^2+\eta_4^2+\eta_5^2+\eta_6^2 $. Now $\eta_1, \eta_2,\eta_5$ are bounded above by the terms on the right hand side by using previous lemma with $\b{v_h} = \b{u_h}$.\\
To bound $\eta_3 $, we have 
    \begin{align*}
    \sum_{e \in\mathcal{E}^0_h}\frac{\eta}{h_e}\|\sjump{\b{u_h}}\|^2_{L^2(e)} &\leq \sum_{e \in\mathcal{E}^0_h}\frac{\eta}{h_e}\|\sjump{\b{ u-u_h}}\|^2_{L^2(e)}+ \sum_{e \in\mathcal{E}^0_h}\frac{\eta}{h_e}\|\sjump{\b{u}}\|^2_{L^2(e)} \\
    &\leq \norm{\b{u-u_h}}^2_h.
    \end{align*}
Further to bound $\eta_4$, we have
    \begin{align*}
     \sum_{e \in \mathcal{E}^C_h}h_e \|\b{\sigma}_{h\tau}(\b{u_h})+{c_\tau} \b{{\lambda_{h\tau}}}\|^2_{L_2(e)} &\leq \sum_{e \in \mathcal{E}^C_h}h_e \|\b{\sigma}_{h\tau}(\b{u_h})+{c_\tau} \b{{\lambda_{\tau}}}\|^2_{L_2(e)}+\sum_{e \in \mathcal{E}^C_h}h_e\|c_\tau (\b{\lambda_{h\tau}-\lambda_{\tau})}\|^2_{L^2(e)} \nonumber\\
    &\leq \sum_{e \in \mathcal{E}^C_h}h_e \|\b{\sigma}_{h\tau}(\b{u_h})+{c_\tau} \b{{\lambda_{\tau}}}\|^2_{L_2(e)}+\sum_{e \in \mathcal{E}^C_h}h_e\|c_\tau\|^2_{L^{\infty}(e)}\|\b{ \lambda_{h\tau}-\lambda_{\tau}}\|^2_{L^2(e)},
    \end{align*}
    therein, a use of last lemma will yield the desired bound.
\par
\noindent
In order to bound $\eta_6$, let $e \in \mathcal{E}_h^C$ be arbitrary. Now, using \eqref{1.7} and triangle inequality, we have
\begin{align*}
    \|{\sigma}_{hn}(\b{u_h})+c_n(u_{hn}-g_a)^{m_n}_+\|_{L^2(e)} &\leq
    \|{\sigma}_{hn}(\b{u_h}-\b{u})\|_{L^2(e)} + \|c_n(u_{hn}-g_a)^{m_n}_+ - c_n(u_{n}-g_a)^{m_n}_+\|_{L^2(e)}.
\end{align*}
A use of \eqref{dm} yields
\begin{align}\label{5.28}
\|{\sigma}_{hn}(\b{u_h})+c_n(u_{hn}-g_a)^{m_n}_+\|_{L^2(e)} &\leq C \|\b{u-u_h}\|_{1,h} + \|{\sigma}_{hn}(\b{u_h}-\b{u})\|_{L^2(e)},
\end{align}
where $C$ is a constant depending on load vectors. Therefore, squaring $\eqref{5.28}$ and summing over all $e \in \mathcal{E}^C_h$ and using the identity $\sum_{e\in \mathcal{E}_h^C} |h_e|=|\Gamma_C|$, we find
\begin{align}\label{5.29}
 \sum_{e \in \mathcal{E}^C_h}h_e\|{\sigma}_{hn}(\b{u_h})+c_n(u_{hn}-g_a)^{m_n}_+\|^2_{L^2(e)}\leq \norm{\b{u-u_h}}^2_h + \sum_{e \in \mathcal{E}^C_h}h_e\|\b{\sigma}_{hn}(\b{u_h-u})\|^2_{L^2(e)} 
\end{align}
This completes the proof.
\end{proof}
\bigbreak
\section{Medius Analysis} \label{sec:Apriori}
In this section, a priori error bounds are derived with minimal regularity assumption on the exact solution $\b{u}$ of \eqref{1.10}, say $\b{u} \in H^{(1+s)}(\Omega)$ for $s \in ( 0,1]$. The name medius analysis indicates that both a priori and a posteriori techniques are employed in this analysis \cite{Gudi:2010:Medius}.
\begin{theorem}
Let $\b{u}$ and $\b{u_h}$ be the solution of continuous problem\eqref{1.10} and discrete problem \eqref{4.1}, respectively. Then, for any $\b{v_h}\in \b{V_h}$, we have 
    \begin{align*}
  \norm{\b{u-u_h}}^2_h &\leq \underset{\b{v_h}\in \b{V_h}}{inf}\Big(\norm{\b{v_{h}-u}}^2_h + \sum_{e \in \mathcal{E}^C_h}h_e^{1/2}\|c_{\tau}\|_{L^\infty(e)}\|\b{u-v_{h}}\|_{L^2(e)}\\&+ \sum_{e \in \mathcal{E}^C_h}h_e\|\b{\sigma}_{hn}(\b{v_h-u})\|^2_{L^2(e)}\Big)+Osc(\b{f})^2+Osc(c_\tau)^2+Osc(\b{\lambda_\tau})^2.
\end{align*}
\end{theorem}
\begin{proof}
Let $\b{v_h}$ be any  arbitrary element in ${\b{V_h}}$. Using triangle inequality and identity $(a+b)^2\leq 2(a^2+b^2)$, we get
    \begin{align}\label{6.1}
    \norm{\b{u-u_h}}^2_h\leq 2(\norm{\b{u-v_h}}^2_h+ \norm{\b{v_h-u_h}}^2_h).
    \end{align}
Setting $\b{\phi}=\b{v_h-u_h}$, and using coercivity of bilinear form $B_h(\cdot,\cdot)$ w.r.t. $\norm{\cdot}_h$, Lemma \ref{ref:Cchar} and equation \eqref{4.1}, we find
\begin{align*}
\alpha \norm{\b{v_h-u_h}}^2_h&\leq B_h(\b{v_h-u_h,v_h-u_h})\\
&\leq B_h(\b{v_h,\phi}) +j_n(\b{u_h,\phi})+j_\tau(\b{v_h})-j_\tau(\b{u_h})-(\b{f},\b{\phi})\\
&\leq B_h(\b{v_h},\b{\phi}-E_h\b{\phi}) + B_h(\b{v_h},E_h\b{\phi})-(\b{f}, \b{\phi}-E_h\b{\phi})-(\b{f},E_h\b{\phi})+j_n(\b{u_h,\phi})\\&+j_\tau(\b{v_h})-j_\tau(\b{u_h})\\
&= B_h(\b{v_h},\b{\phi}-E_h\b{\phi})+B_h(\b{v_h},E_h\b{\phi})- (\b{f}, \b{\phi}-E_h\b{\phi})+j_n(\b{u_h,\phi}) +j_\tau(\b{v_h})\\&-j_\tau(\b{u_h})
-a(\b{u},E_h\b{\phi})-g(\b{\lambda_\tau},E_h\b{\phi})-j_n(\b{u},E_h\b{\phi})\\
&= B_h(\b{v_h},\b{\phi}-E_h\b{\phi})- (\b{f}, \b{\phi}-E_h\b{\phi})+g(\b{\lambda_\tau,\phi}-E_h\b{\phi})+j_n(\b{v_h,\phi}-E_h\b{\phi})\\
&+B_h(\b{v_h},E_h\b{\phi})-a(\b{u},E_h\b{\phi})+j_\tau(\b{v_h})-j_\tau(\b{u_h})-g(\b{\lambda_\tau,\phi}) +j_n(\b{u_h,\phi})\\&-j_n(\b{u},E_h\b{\phi})-j_n(\b{v_h,\phi}-E_h\b{\phi})\\
&= R_1+ R_2+ R_3+ R_4,
\end{align*}
where, \begin{align*}
R_1 &= B_h(\b{v_h},\b{\phi}-E_h\b{\phi})- (\b{f}, \b{\phi}-E_h\b{\phi})+g(\b{\lambda_\tau,\phi}-E_h\b{\phi})+j_n(\b{v_h,\phi}-E_h\b{\phi}),\\
R_2 &= B_h(\b{v_h,E_h\b{\phi}})-a(\b{u},E_h\b{\phi}),\\
R_3 &= j_\tau(\b{v_h})-j_\tau(\b{u_h})-g(\b{\lambda_\tau,\phi}),\\
R_4 &= j_n(\b{u_h,\phi})-j_n(\b{u},E_h\b{\phi})-j_n(\b{v_h,\phi}-E_h\b{\phi}).
\end{align*}
Now, we will estimate $R_1$, $R_2$, $R_3$ and $R_4$ one by one. In order to estimate $R_1$, let $\b{\xi}=\b{\phi}-E_h\b{\phi} $ and thereafter using integration by parts in the first term and gather the resulting terms, we find
   \begin{align*}
  R_1 &= a_h(\b{v_h,\xi})+b_h(\b{v_h,\xi})- (\b{f}, \b{\xi})+g(\b{\lambda_\tau,\b{\xi}})+j_n(\b{v_h,\b{\xi}})\\
   &=\sum_{T\in \mathcal{T}_h}\int_{T}\b{\sigma}(\b{v_h}):\b{\epsilon}(\b{\xi})~dx+ b_h(\b{v_h,\xi})- (\b{f}, \b{\xi})+g(\b{\lambda_\tau,\b{\xi}})+j_n(\b{v_h,\b{\xi}})\\
   &=\sum_{e\in \mathcal{E}_h^i}\int_{e}\sjump{\b{\sigma}_h(\b{v_h})}\cdot\smean{\b{\xi}}~ds  \nonumber +\sum_{e\in \mathcal{E}_h^0}\int_{e}\smean{\b{\sigma}_h(\b{v_h)}}:\sjump{\b{\xi}}~ds
     +\sum_{e\in \mathcal{E}_h^F \cup \mathcal{E}_h^C \ }\int_{e}\b{\sigma}_h(\b{v_h)}\b{n_e}\cdot\b{\xi}~ds \\&+  b_h(\b{v_h,\xi})-(\b{f}, \b{\xi})+g(\b{\lambda_\tau,\b{\xi}})+j_n(\b{v_h,\b{\xi}})\\
      &= -\sum_{T\in \mathcal{T}_h} \int_{T} \b{f\cdot \xi}~dx
     +\sum_{e \in \mathcal{E}_h^F}\int_{e}(\b{\sigma}_h(\b{v_h})\b{n_e}-\b{g})\cdot \b{\xi}~ds + \sum_{e\in \mathcal{E}_h^i}\int_{e}\sjump{\b{\sigma}_h(\b{v_h}}\cdot
     \smean{\b{\xi}}~ds\\ &+\sum_{e\in \mathcal{E}_h^0}\int_{e}\smean{\b{\sigma}_h(\b{v_h})}:\sjump{\b{\xi}}~ds +b_h(\b{v_h,\xi}) + \sum_{e\in \mathcal{E}_h^C}\int_{e}(\b{\sigma}_{hn}(\b{v_{h}})+ c_{n}(v_{hn}-g_a)_+^{m_n}) \cdot \xi_n~ds \\&+\sum_{e \in \mathcal{E}_h^C} \int _e (\b{\sigma}_{h\tau}(\b{v_h})+c_\tau \b{\lambda_{\tau})\cdot \xi_\tau}~ds.
   \end{align*}
It can be observed that the following estimate holds for all the DG methods introduced in Section 3
   \begin{align}\label{6.2}
    \sum_{e\in \mathcal{E}_h^0}\int_{e}\smean{\b{\sigma}_h(\b{v_h})}:\sjump{\b{\xi}}~ds +b_h(\b{v_h,\xi}) \leq \bigg(\sum_{e \in E_h^0} \int_{e}\frac{1}{h_e}\sjump{\b{v_h}}^2~ds\bigg)^{1/2}\norm{\b{\phi}}_h.
   \end{align}
Using \eqref{6.2} and the similar arguments used in Theorem 4.1, we obtain
   \begin{align*}
    R_1\leq \eta(\b{v_h}) \norm{\b{\phi}}_h,
    \end{align*}
where,
    \begin{align*}
     \eta(\b{v_h})^2 =& \sum_{T \in \mathcal{T}_h} h_T^2\|\b{f}\|^2_{L^2(T)}+\sum_{e \in \mathcal{E}_h^i} h_e\|\sjump{\b{\sigma}_h(\b{v_h})}~ \|^2_{L^2(e)}+ \sum_{e \in \mathcal{E}_h^0} \frac{\eta}{h_e}\|\sjump{\b{v_h}}\|^2_{L^2(e)}\\
     &+\sum_{e \in \mathcal{E}_h^C} h_e\|\b{\sigma}_{h\tau}(\b{v_h})+  c_{\tau}\b{\lambda_{\tau}}\|^2_{L^2(e)}+ \sum_{e \in \mathcal{E}_h^F} h_e\|\b{\sigma}_h(\b{v_h})\b{n}-\b{g}\|^2_{L^2(e)}\\ &+
     \sum_{e \in \mathcal{E}_h^C} h_e\|\b{\sigma}_{hn}(\b{v_h})+c_n(v_{hn}-g_a)^{m_n}_+\|^2_{L^2(e)}.
     \end{align*}
A use of Young's inequality and Lemma 4.2 yields
\begin{align*}
    R_1 &\leq \frac{1}{\beta}(\norm{\b{v_h-u}}_h^2 +  \|\b{v_h-u}\|^{2m_n}_{1,h} + Osc(\b{f})^2+Osc(c_\tau)^2+Osc(\b{\lambda_\tau})^2\\&+ \sum_{e \in \mathcal{E}^C_h}h_e\|\b{\sigma}_{hn}(\b{v_h-u})\|^2_{L^2(e)}) + \beta\norm{\b{\phi}}_h^2.
\end{align*}
where $\beta>0$ is arbitrary. Using the definition of $a(\cdot,\cdot)$ and $a_h(\cdot, \cdot)$, (3.2), Lemma \ref{lem:EO} and Young's inequality, the bound on $R_2$ can be obtained as follows: 
    \begin{align*}
      R_2 &= B_h(\b{v_h,E_h\b{\phi}})-a(\b{u},E_h\b{\phi})\\ 
    &=a_h(\b{v_h,E_h\b{\phi}})+b_h(\b{v_h,E_h\b{\phi}})-a(\b{u},E_h\b{\phi})\\
     &=\sum_{T \in \cT_h} |\b{u-v_h}|_{H^1(T)}|E_h\b{\phi}|_{H^1(T)}+\bigg(\int_{\mathcal{E}_h^0}h_e^{-1}\sjump{\b{v_h-u}}^2~ds \bigg)^{1/2}|E_h\b{\phi}|_{H^1(\Omega)}\\
     &\leq \norm{\b{u-v_h}}_h~|E_h\b{\phi}|_{H^1(\Omega)}\\
     &\leq \norm{\b{u-v_h}}_h~\norm{\b{\phi}}_h\\
     &\leq \frac{1}{\beta_1}\norm{\b{u-v_h}}_h^2 + \beta_1\norm{\b{\phi}}_h^2,
     \end{align*}
where $\beta_1>0$ is arbitrary.
In order to estimate $R_3$, we will use $\b{u_\tau}\cdot \b{\lambda_\tau}=|\b{u_\tau}|$ and $|\b{\lambda_\tau}| \leq 1$ a.e. on $\Gamma_C$  and find
    \begin{align*}
    R_3 &= j_\tau(\b{v_h})-j_\tau(\b{u_h})-g(\b{\lambda_\tau,\phi})\\
     &= \int_{\Gamma_C}c_{\tau}|\b{v_{h\tau}}|~ds-\int_{\Gamma_C}c_{\tau}|\b{u_{h\tau}}|~ds-\int_{\Gamma_C}c_{\tau}\b{\lambda_\tau} \cdot \b{v_{h \tau}}~ds+ \int_{\Gamma_C}c_{\tau}\b{\lambda_\tau} \cdot \b{u_{h \tau}}~ds\\
    &\leq \int_{\Gamma_C}c_{\tau}|\b{v_{h\tau}}|~ds-\int_{\Gamma_C}c_{\tau}\b{\lambda_\tau} \cdot \b{v_{h \tau}}~ds\\
     &=\int_{\Gamma_C}c_{\tau}\big(|\b{v_{h\tau}}|-|\b{u_\tau}| \big)~ds+\int_{\Gamma_C}c_{\tau}\b{\lambda_\tau} \cdot \big(\b{u_\tau-v_{h \tau}})~ds\\
     &\leq 2\int_{\Gamma_C}c_{\tau}|\b{u_\tau-v_{h\tau}}|~ds\\
    &\leq 2\sum_{e\in \mathcal{E}^C_h}\|c_\tau\|_{L^\infty(e)}h_e^{1/2}\|\b{u-v_h}\|_{L^2(e)}.
     \end{align*}
A use of monotoncity argument \cite{Han:1993:FCPNC}, Cauchy Holder's inequality, identity $|(a)^m_+-(b)^m_+|\leq m|a-b|\big(|a|^{m-1}+|b|^{m-1}\big)~ a,~b \in \mathbb{R}, m\geq1$, \eqref{im} and Lemma 2.4 in $R_4$ yields
    \begin{align*}
     R_4 &= j_n(\b{u_h,\phi})-j_n(\b{u},E_h\b{\phi})-j_n(\b{v_h,\phi}-E_h\b{\phi})\\
     &= j_n(\b{u_h,\b{v_h-u_h}})-j_n(\b{u},E_h\b{\phi})-j_n(\b{v_h,\phi}-E_h\b{\phi})+j_n(\b{v_h,\b{v_h-u_h}})\\&-j_n(\b{v_h,\b{v_h-u_h}})\\
      &\leq j_n(\b{v_h,\b{v_h-u_h}})-j_n(\b{v_h,\phi}-E_h\b{\phi})-j_n(\b{u},E_h\b{\phi})\\
      &=j_n(\b{v_h,\b{\phi}})-j_n(\b{v_h,\phi}-E_h\b{\phi})-j_n(\b{u},E_h\b{\phi})\\
      &=j_n(\b{v_h},{E_h\b{\phi}})-j_n(\b{u},E_h\b{\phi})\\
      &= \int_{\Gamma_C}((c_n(v_{hn}-g_a)^{m_n}_+ - c_n(u_{n}-g_a)^{m_n}_+) (E_{h}\b{\phi})_n~ds\\
      &\leq \| c_n(v_{hn}-g_a)^{m_n}_+ -c_n(u_n-g_a)^{m_n}_+ \|_{L^2(\Gamma_C)}\|(E_{h}\b{\phi})_n\|_{L^2(\Gamma_C)}\\
      &\leq \| c_n\|_{L^{\infty}{(\Gamma_C)}}\|(v_{hn}-g_a)^{m_n}_+ - (u_n-g_a)^{m_n}_+ \|_{L^2(\Gamma_C)}\|E_{h}\b{\phi}\|_{H^1(\Omega)}\\
      &\leq\| c_n\|_{L^{\infty}{(\Gamma_C)}}\|(v_{hn}-g_a)^{m_n}_+ - (u_n-g_a)^{m_n}_+ \|_{L^2(\Gamma_C)}\norm{\b{\phi}}_h.
      \end{align*}
Following the similar arguments,  used in proving $(v)$ of Lemma 4.2, we obtain
   \begin{align*}
    R_4 &\leq C\bigg(\|\b{v_{h}-u}\|^{m_n}_{1,h}+\|\b{v_{h}-u}\|_{1,h}\bigg)\norm{\b{\phi}}_h.
    \end{align*}
Further, Young's inequality yields
    \begin{align*}
    R_4 &\leq C\bigg(\|\b{v_{h}-u}\|^{2m_n}_{1,h}+\|\b{v_{h}-u}\|^2_{1,h}\bigg)+ {\beta_2}\norm{\b{\phi}}^2_h.
    \end{align*}
where $\beta_2>0$ is arbitrary.
Combining the bounds on $R_1$, $R_2$, $R_3$ and $R_4$, and choosing $\beta, \beta_1$ and $\beta_2$ sufficiently small, we obtain
\begin{align}\label{6.3}
   \norm{\b{v_h-u_h}}^2_h &\leq C\bigg( \|(\b{v_{h}-u})\|^{2m_n}_{1,h}+\norm{\b{v_{h}-u}}^2_h +\sum_{e \in \mathcal{E}_h^C}h_e^{1/2}\|c_{\tau}\|_{L^{\infty}(e)}\|\b{u-v_{h}}\|_{L^2(e)} \nonumber \\&+ \sum_{e \in \mathcal{E}^C_h}h_e\|\b{\sigma}_{hn}(\b{v_h-u})\|^2_{L^2(e)}+ Osc(\b{f})^2+Osc(c_\tau)^2+Osc(\b{\lambda_\tau})^2 \bigg).
\end{align}
Thus, using $\eqref{6.3}$ in $\eqref{6.1}$, we obtain
\begin{align*}
    \norm{\b{u-u_h}}^2_h &\leq C \bigg(\|(\b{v_{h}-u})\|^{2m_n}_{1,h}+\norm{\b{v_{h}-u}}^2_h +\sum_{e \in \mathcal{E}_h^C} h_e^{1/2}\|c_{\tau}\|_{L^{\infty}(e)}\|\b{u-v_{h}}\|_{L^2(e)} \nonumber \\  &+\sum_{e \in \mathcal{E}^C_h}h_e\|\b{\sigma}_{hn}(\b{v_h-u})\|^2_{L^2(e)}+ Osc(\b{f})^2+Osc(c_\tau)^2+Osc(\b{\lambda_\tau})^2 \bigg) \\
    &\leq C\underset{\b{v_h}\in \b{V_h}}{inf}\Big(\|\b{v_{h}-u}\|^{2m_n}_{1,h}+\norm{\b{v_{h}-u}}^2_h +\sum_{e \in \mathcal{E}^C_h}h_e^{1/2}\|c_{\tau}\|_{L^{\infty}(e)}\|\b{u-v_{h}}\|_{e}\\&+ \sum_{e \in \mathcal{E}^C_h}h_e\|\b{\sigma}_{hn}(\b{v_h-u})\|^2_{L^2(e)}\Big)+Osc(\b{f})^2+Osc(c_\tau)^2+Osc(\b{\lambda_\tau})^2.
    \end{align*}
  Since $m_n \geq 1$ which implies $2 m_n\geq 2$, therefore
\begin{align*}
    \underset{\b{v^h}\in \b{V_h}}{inf}\|\b{v_h-u}\|_{1,h}^{2m_n} \leq C \underset{\b{v_h}\in \b{V_h}}{inf}\|\b{v_h-u}\|_{1,h}^{2}.
\end{align*}
Hence,
\begin{align*}
  \norm{\b{u-u_h}}^2_h &\leq \underset{\b{v_h}\in \b{V_h}}{inf}\Big(\norm{\b{v_{h}-u}}^2_h +\sum_{e \in \mathcal{E}^C_h}h_e^{1/2}\|c_{\tau}\|_{L^{\infty}(e)}\|\b{u-v_{h}}\|_{L^2(e)}\\&+ \sum_{e \in \mathcal{E}^C_h}h_e\|\b{\sigma}_{hn}(\b{v_h-u})\|^2_{L^2(e)}\Big)+Osc(\b{f})^2+Osc(c_\tau)^2+Osc(\b{\lambda_\tau})^2.
\end{align*}
\end{proof}
The following result is a consequence of the last theorem with the choice of $\b{v_h}$ as in Lemma \ref{lem:Interpolation}.
\begin{theorem} \label{thm:apriori}
Suppose $\b{u} \in H^{(1+s)}(\Omega)$ for some $s \in ( 0,1]$. Then, there exists a constant $C>0$, depending in the shape regularity of $\mathcal{T}_h$ such that
\begin{align*}
\norm{\b{u-u_h}}_h &\leq C h^s.
\end{align*}
\end{theorem}

\begin{remark}
If the regularity of the continuous solution $\b{u}$ is $[H^2(\O)]^2$, then from Theorem \ref{thm:apriori} the  error term $\norm{\b{u-u_h}}_h$ converges with linear rate which is optimal. If $\b{u} \in [H^1(\O)]^2$,  one can easily show that $\norm{\b{u − v_h}}_h \rightarrow 0$ as $h \rightarrow 0$ using the density argument \cite{Ciarlet:1978:FEM}. Hence, whenever $h \rightarrow 0$, then $\norm{\b{u − u_h}}_h \rightarrow 0$.
\end{remark}

\begin{remark}
The abstract error estimate in Theorem 5.1 also holds for $m_n=0$.
\end{remark}
\section{Numerical Experiments} \label{sec:NumR}
In this section, we carry out the numerical experiments to illustrate the performance of a posteriori estimator derived in the Section \ref{sec:Apost} as well as the convergence behaviour of error on uniform meshes. In order to perform numerical experiments we have implemented the codes in Matlab 9.8.0 (R2020a). Uzawa algorithm \cite{Glowinski:2008:VI} is used to solve the discrete problem, therein we set $10^{-8}$ to be the relative error tolerance in the maximum norm. 

\par
\noindent
For illustrating the behaviour of error estimator, we use the following algorithm:
\begin{center}
 \textbf{SOLVE} $\longrightarrow$ \textbf{ESTIMATE} $\longrightarrow$ \textbf{MARK} $\longrightarrow$ \textbf{REFINE}
\end{center}

In the step \textbf{SOLVE},  the discrete problem is solved for $\b{u_h}$. Then, the error estimator $\eta_h$ is computed on each element in the  step \textbf{ESTIMATE} and D$\ddot{o}$rfler marking scheme \cite{Dorfler:1996:Afem} with parameter $\theta = 0.3$ is used in the step \textbf{MARK}. Finally in the last step \textbf{REFINE}, the marked elements undergo refinement using the newest vertex bisection algorithm and the above algorithm is repeated.
\par
Now, we present  numerical results for two test examples solved by SIPG and NIPG method. As the exact solution $\b{u}$ is unknown in both examples, error on uniform mesh is computed by calculating the difference between the discrete solutions $\b{u_h}$ obtained on the consecutive mesh.
 In these examples the Lam$\acute{e}$ parameters $\mu$ and $\lambda$ are computed by
\begin{align*}
\mu = \frac{E}{2(1+\nu)}~~ \text{and}~~\lambda = \frac{E\nu}{(1+\nu)(1-2\nu)}
\end{align*}
where, $E$ and $\nu$ denote the Young’s modulus and the Poisson ratio, respectively. For both the examples the penalty parameter $\eta$ is chosen to be $30\mu$. 
\begin{example} In this example, 
we consider the domain $\O$ as $(0,1)\times(0.05,1.05)$ and the following data (the unit $daN/mm^2$ stands for “decaNewtons per square millimeter”):
\begin{align*}
\Gamma_D &= \{1\}\times (0.05,1.05),
\\
\Gamma_F &= (\{0\}\times (0.05, 1.05))\cup ((0,1)\times \{1.05\}),\\
\Gamma_C &= (0,1)\times \{0.05\},\\
E &= 2000daN/mm^2,~~ \nu=0.4~~, \b{f}=(0,0)daN/mm^2,\b{g}=(200(5-y),-190)daN/mm^2,\\
c_\tau&=450,~ c_n=1,~m_n=1,~ g_a=0.05~mm.
\end{align*}
\end{example}
The convergence behavior of error for  SIPG and NIPG methods on the uniform mesh is shown in Table 6.1. Figure 6.1 describes the behavior of the  residual estimators for SIPG and NIPG methods, respectively on adaptive meshes. We observe that the estimator converges optimally on the adaptive mesh.
Figure 6.2 show the adaptive mesh refinement at a certain level for  SIPG and NIPG method. We observe the mesh is refined more near the intersection of the boundaries and near the contact edge, as it is evident that the body undergoes deformation under the action of traction. Hence, the singular behavior of the discrete solution is well captured by the estimator.
\begin{table}
\begin{tabular}{ |c|c|c| } 
 \hline
 ${h}$ & error & order of conv. \\
 \hline
 $2^{-1}$ & 4.7518 $\times$ $10^{-1}$ & -\\ 
 $2^{-2}$ & 2.9620 $\times$ $10^{-1}$  & 0.6818\\ 
 $2^{-3}$ & 1.7708 $\times$ $10^{-1}$ & 0.7421 \\ 
 $2^{-4}$ & 1.0731 $\times$ $10^{-1}$  & 0.7502\\
 $2^{-5}$ & 6.7152 $\times$ $10^{-2}$ & 0.7913 \\
 \hline
\end{tabular}
\quad
~~~~
\begin{tabular}{ |c|c|c| } 
 \hline
 ${h}$ & error & order of conv.\\
 \hline
 $2^{-1}$ & 4.8844 $\times$ $10^{-1}$ & -\\ 
 $2^{-2}$ & 3.0242 $\times$ $10^{-1}$  & 0.6916\\ 
 $2^{-3}$ & 1.807 $\times$ $10^{-1}$ & 0.7427 \\ 
 $2^{-4}$ & 1.0732 $\times$ $10^{-1}$  & 0.7519\\
 $2^{-5}$ & 6.7186 $\times$ $10^{-2}$ & 0.7934 \\ 
 \hline
\end{tabular}
\bigbreak
\caption{Errors and orders of
convergence for SIPG and NIPG methods on uniform mesh for Example 1}
\end{table}
\begin{figure*}[!ht]
	\begin{minipage}{7cm}
			
	\includegraphics[width=1.1\textwidth]{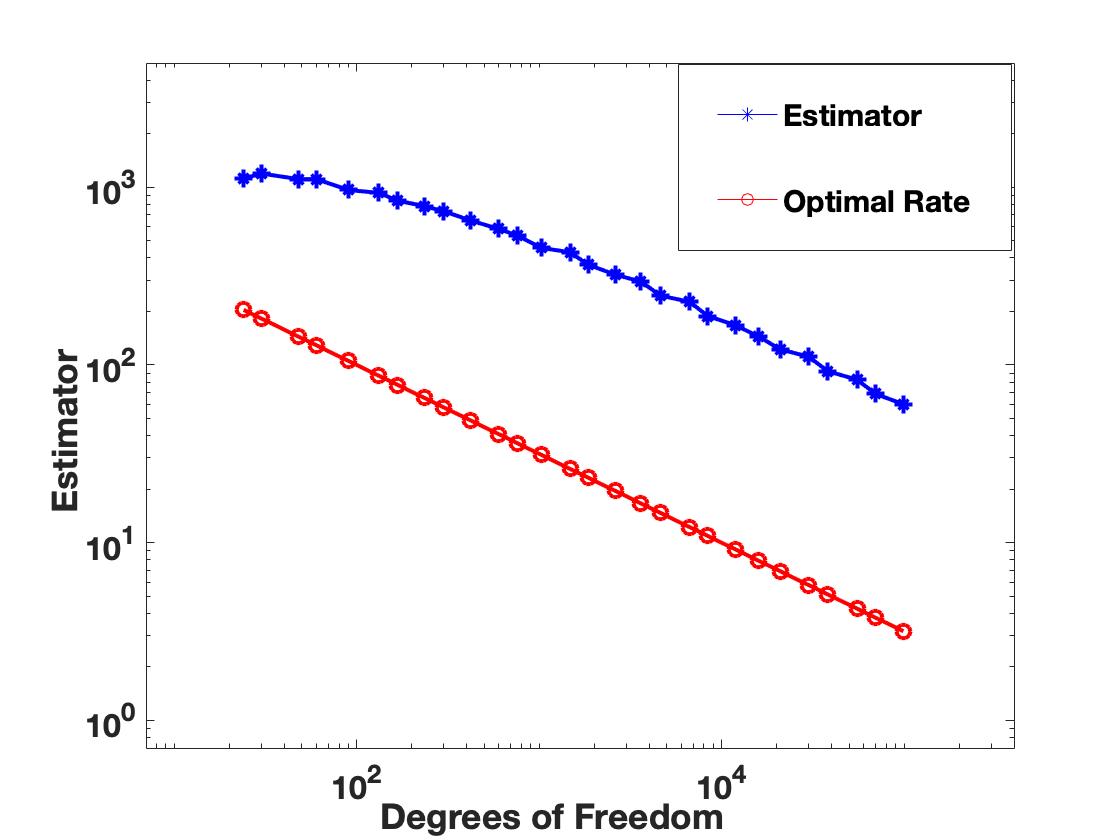}
	
	\end{minipage}
	\hspace{0.1cm}
	\begin{minipage}{7cm}		
	
	\includegraphics[width=1.1\textwidth]{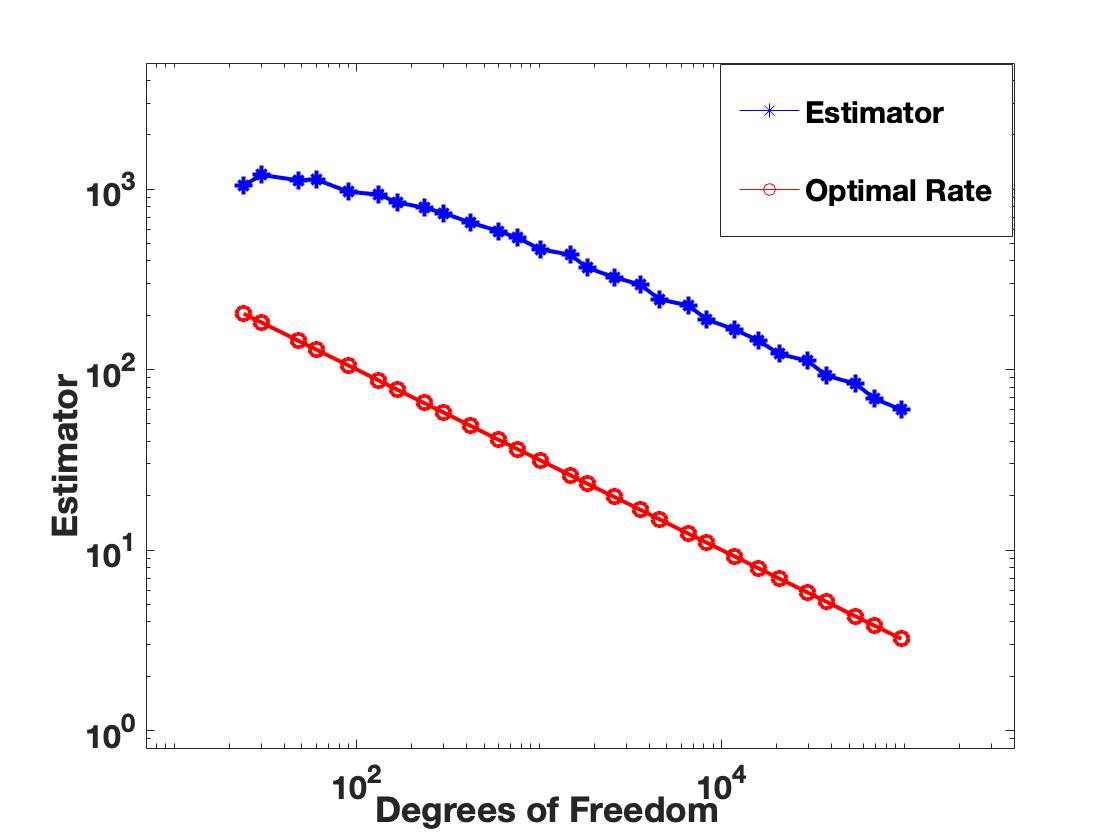}	

	\end{minipage}
\caption{Estimator for SIPG and NIPG method for Example 1}
	\label{0}
\end{figure*}

\begin{figure*}[!ht]
	\begin{minipage}{7cm}
			
	\includegraphics[width=1.1\textwidth]{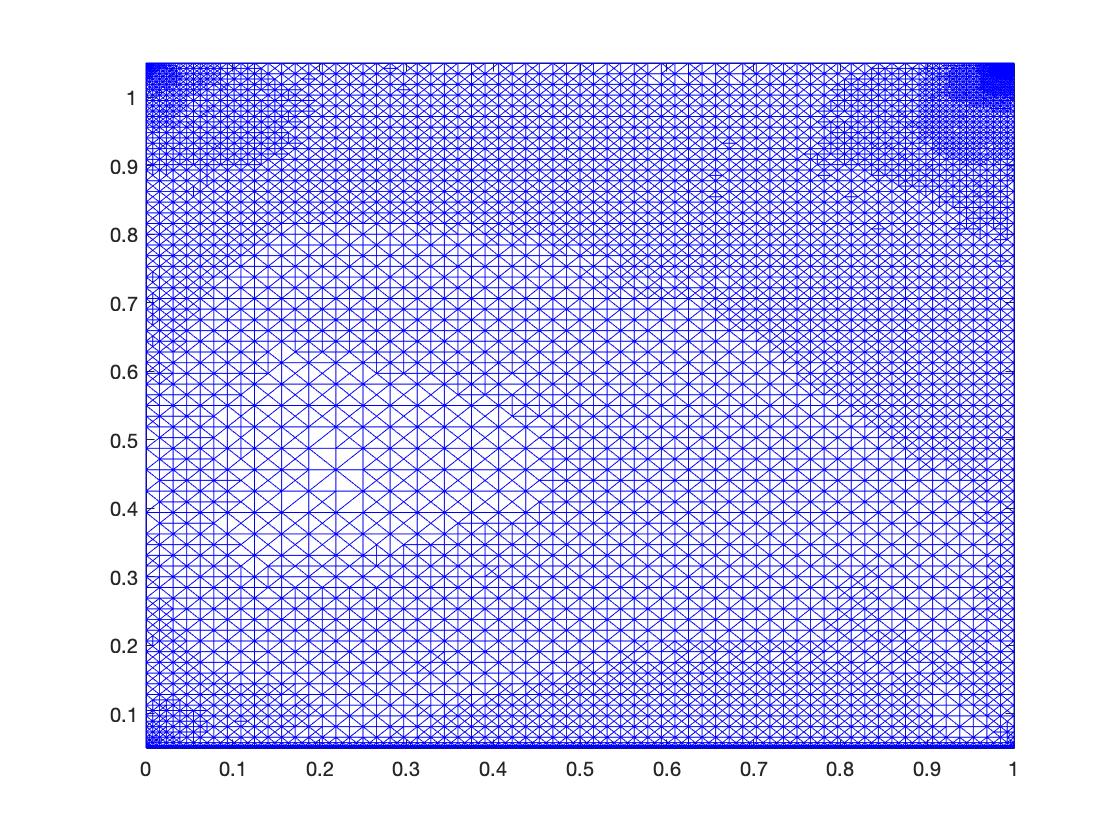}
	
	\end{minipage}
	\hspace{0.1cm}
	\begin{minipage}{7cm}		
	
	\includegraphics[width=1.1\textwidth]{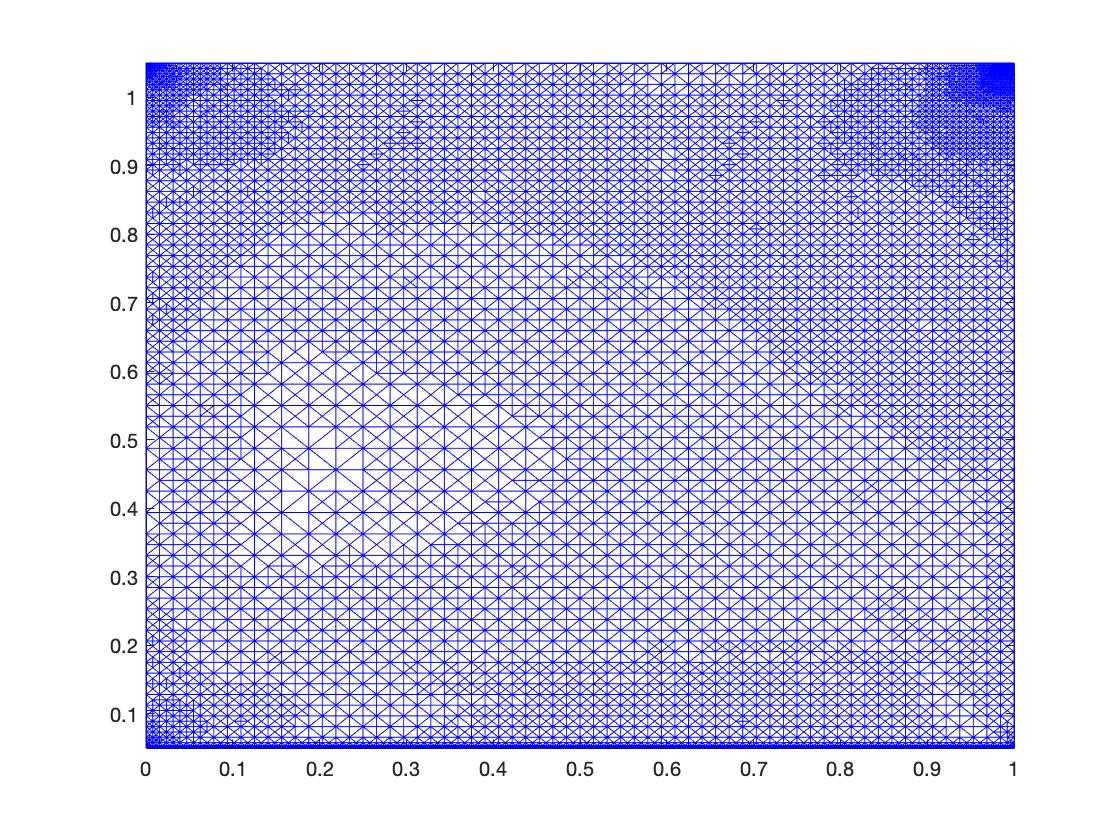}	

	\end{minipage}
\caption{Adaptive mesh for SIPG and NIPG methods for Example 1 at level 28}
	\label{2}
\end{figure*}

\begin{example} Therein, we consider the domain $\O$ as $(0,1)\times(0,1)$ together with the following data:
\begin{align*}
\Gamma_D &= (0,1)\times \{1\},
\\
\Gamma_F &= (\{0\}\times (0, 1))\cup (\{1\}\times (0, 1)),\\
\Gamma_C &= (0,1)\times \{0\},\\
E &= 2500daN/mm^2,~~ \nu=0.2~~, \b{f}=(0,0)daN/mm^2,\\
\b{g}&=(880, 0)daN/mm^2,~ c_\tau=250,~ c_n=1,~ g_a=0.00mm,~m_n=1.
\end{align*}
\end{example}
\begin{table}
\begin{tabular}{ |c|c|c| } 
 \hline
 ${h}$ & error & order of conv. \\
 \hline
 $2^{-1}$ & 6.7573 $\times$ $10^{-1}$ & -\\ 
 $2^{-2}$ & 4.1330 $\times$ $10^{-1}$  & 0.7091\\ 
 $2^{-3}$ & 2.4171 $\times$ $10^{-1}$ & 0.7735 \\ 
 $2^{-4}$ & 1.4053 $\times$ $10^{-1}$  & 0.7824\\
 $2^{-5}$ & 8.235 $\times$ $10^{-2}$ & 0.7901 \\
 \hline
\end{tabular}
\quad
~~~~
\begin{tabular}{ |c|c|c| } 
 \hline
 ${h}$ & error & order of conv.\\
 \hline
 $2^{-1}$ & 6.7731 $\times$ $10^{-1}$ & -\\ 
 $2^{-2}$ & 4.1610 $\times$ $10^{-1}$  & 0.7028\\ 
 $2^{-3}$ & 2.4403 $\times$ $10^{-1}$ & 0.7698 \\ 
 $2^{-4}$ & 1.4197 $\times$ $10^{-1}$  & 0.7814\\
 $2^{-5}$ & 8.241 $\times$ $10^{-2}$ & 0.7896 \\ 
 \hline
\end{tabular}
\bigbreak
\caption{Errors and orders of
convergence for SIPG and NIPG methods on uniform mesh for Example 2}
\end{table}
Table 6.2 depicts the errors and orders of
convergence behavior of SIPG and NIPG methods on uniform mesh for Example 2. Figure 6.3 describes the behaviour of the residual estimators for SIPG and NIPG methods, with the increasing degree of freedom on adaptive meshes. Clearly, the estimator converges optimally on the adaptive mesh. Figure 6.4 show the adaptive mesh refinement at level 23 for the SIPG and NIPG method. We observe that the mesh refinement is high near the contact edge due to the effect of traction and near the  corners due to the intersection of boundaries.


\begin{figure*}[!ht]
	\begin{minipage}{7cm}
			
	\includegraphics[width=1.1\textwidth]{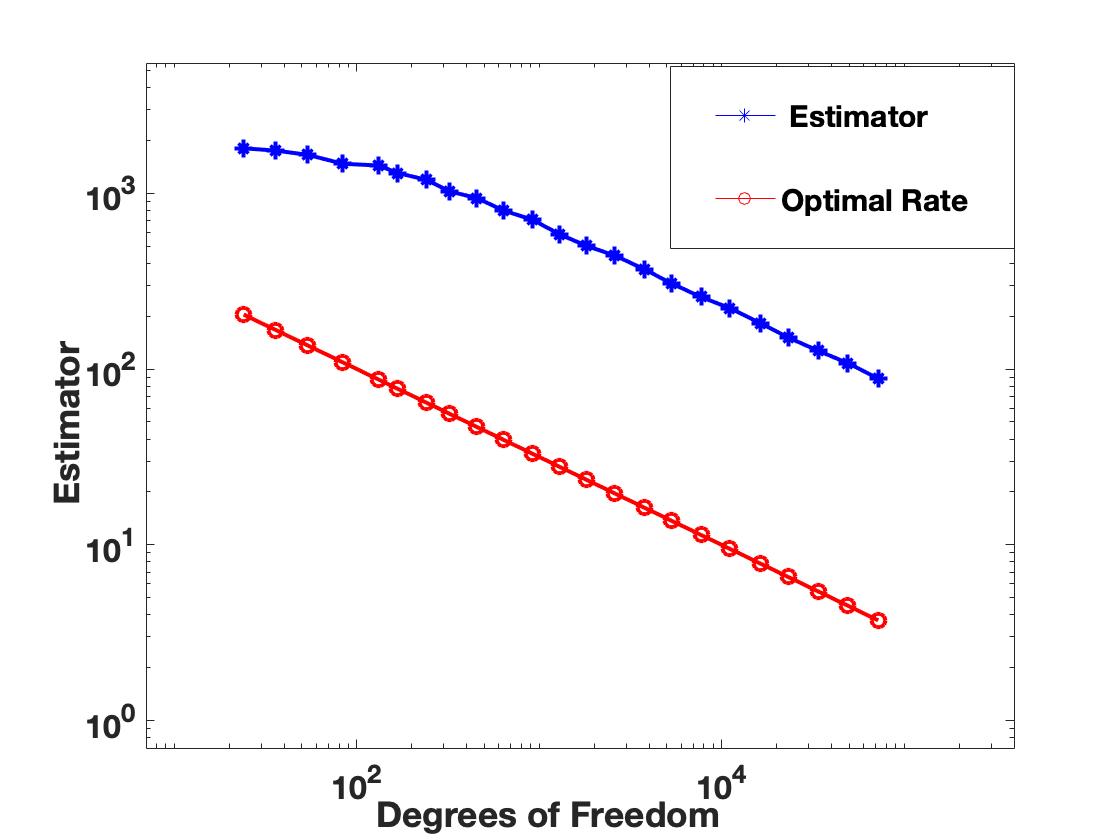}
	
	\end{minipage}
	\hspace{0.1cm}
	\begin{minipage}{7cm}		
	
	\includegraphics[width=1.1\textwidth]{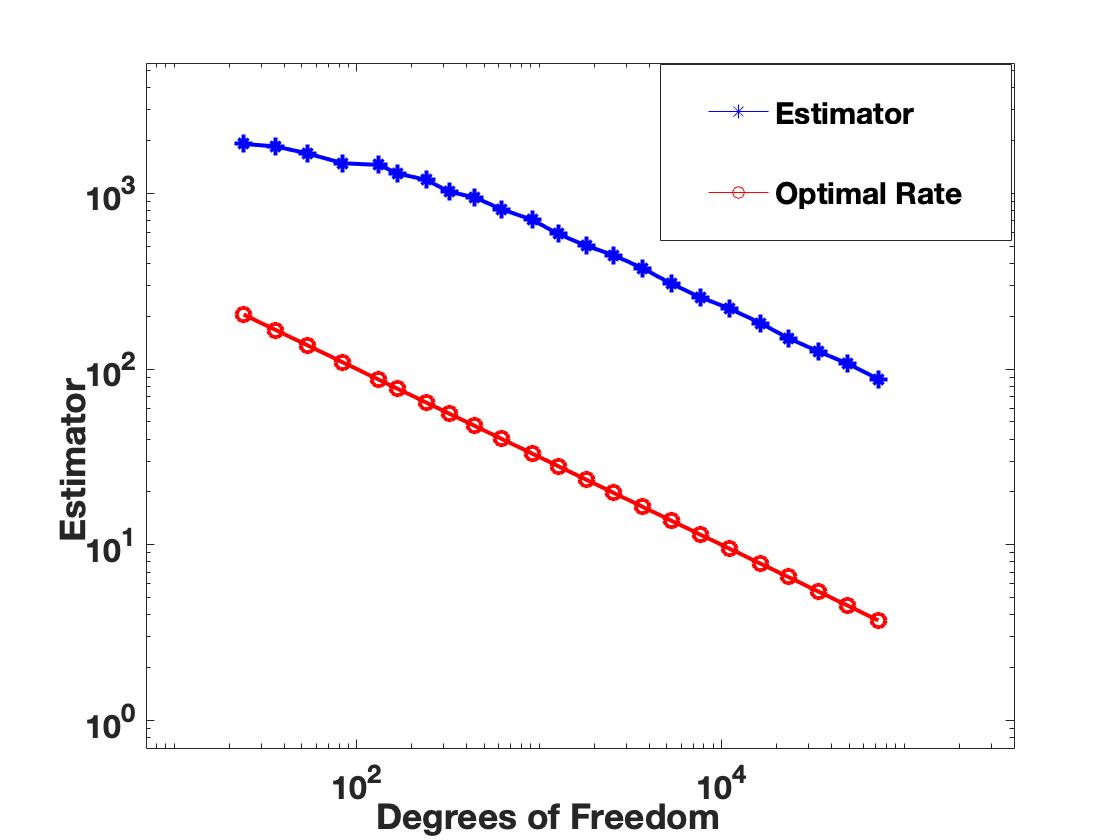}	

	\end{minipage}
\caption{Estimator for SIPG and NIPG method for Example 2}
	\label{4}
\end{figure*}
\begin{figure*}[!ht]
	\begin{minipage}{7cm}
			
	\includegraphics[width=1.1\textwidth]{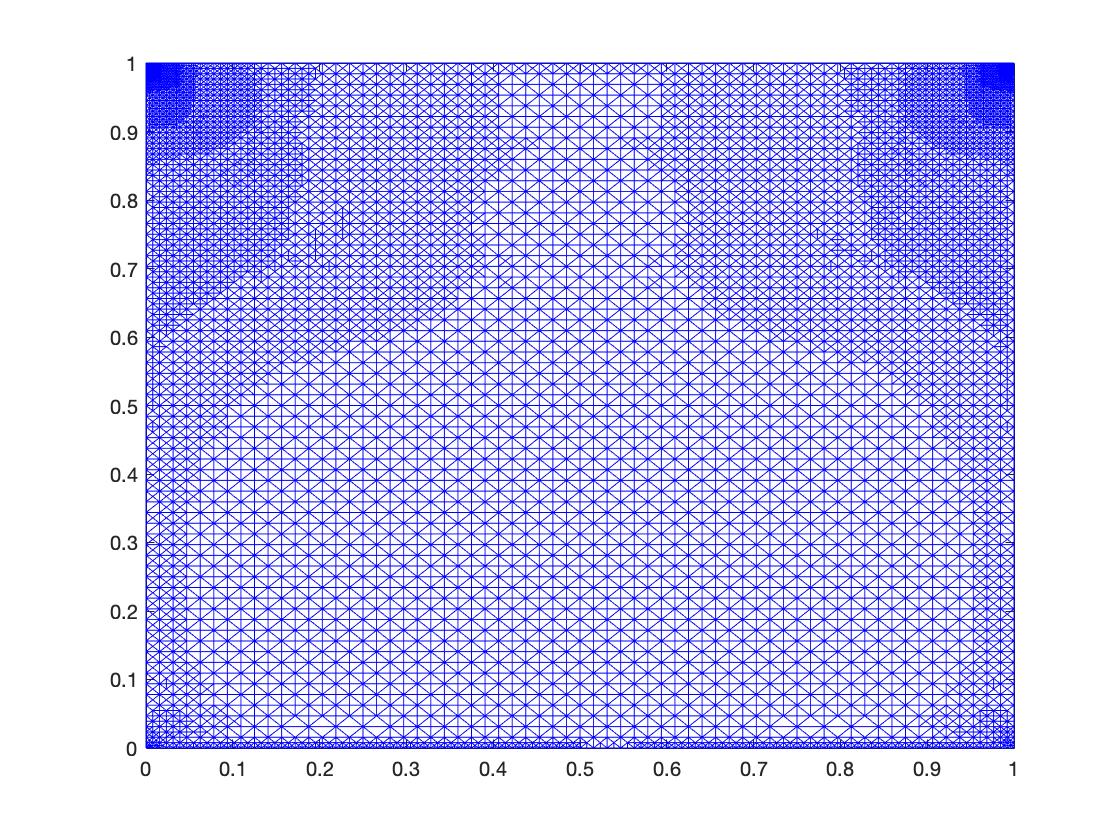}
	
	\end{minipage}
	\hspace{0.1cm}
	\begin{minipage}{7cm}		
	
	\includegraphics[width=1.1\textwidth]{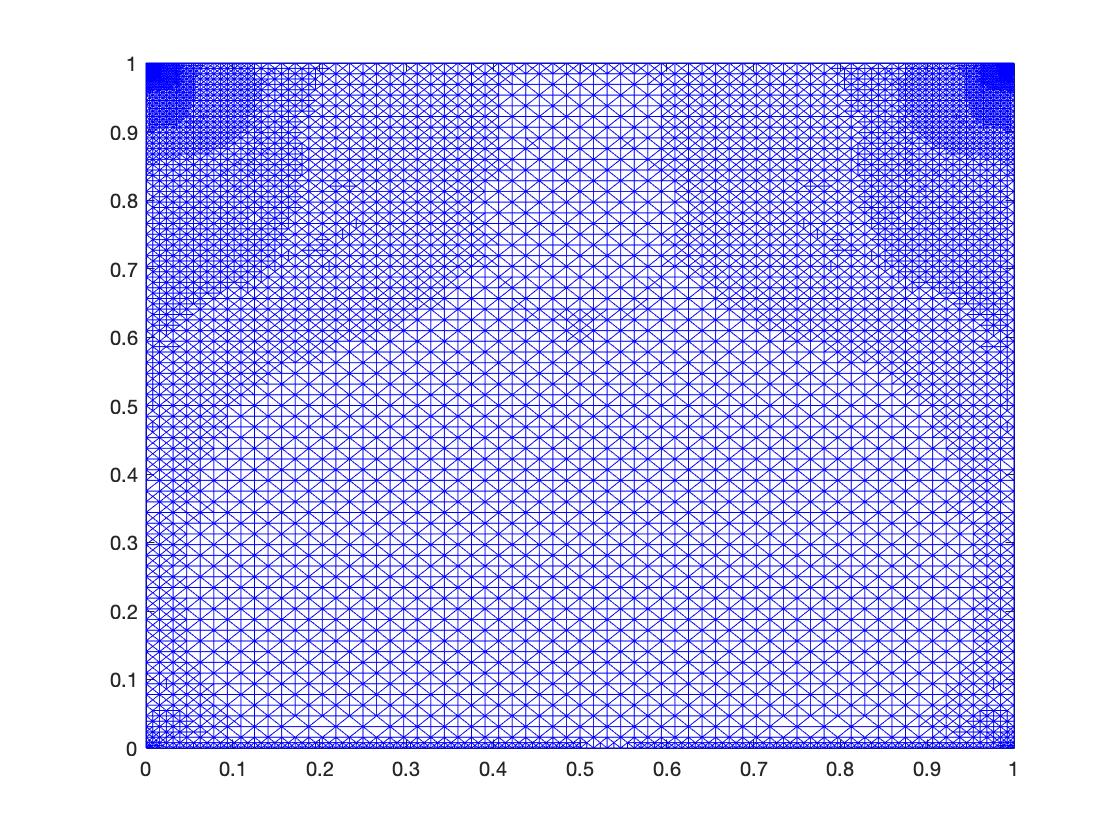}	

	\end{minipage}
\caption{Adaptive mesh for SIPG and NIPG methods for Example 2 at level 23}
	\label{6}
\end{figure*}
\section{Conclusions} \label{sec:Summ}
In this paper, we have derived residual based a posteriori error estimators for a class of DG methods for frictional contact problem with reduced normal compliance. The reliability and the efficiency of  a posteriori error estimator has been discussed. An abstract a priori error estimate has been derived assuming minimal regularity on the exact solution $\b{u}$. Numerical results are presented to demonstrate the convergence behaviour over uniform as well as adaptive mesh. The results of this article are also valid for conforming finite element methods. The case with $m_t>0$ will be addressed in future.

\end{document}